\documentclass[a4paper,11pt]{article}
\usepackage{amsmath,amssymb,amsthm,MnSymbol}
\usepackage{paralist,enumitem,hyperref,url,bm,placeins}
\usepackage{color,graphicx}
\usepackage[margin=30mm]{geometry}

\newcommand{\bv}{\bm{v}}
\newcommand{\bc}{\bm{c}}

\newcommand{\bz}{\bm{z}}
\newcommand{\bq}{\bm{q}}
\newcommand{\bxi}{\bm{\xi}}
\newcommand{\by}{\bm{y}}

\newcommand{\bmu}{\bm{\mu}}

\newcommand{\bu}{\bm{u}}

\AtBeginDocument{
  \setlength{\abovedisplayskip}{4pt}
  \setlength{\belowdisplayskip}{4pt}
  \setlength{\abovedisplayshortskip}{2pt}
  \setlength{\belowdisplayshortskip}{2pt}
}

\theoremstyle{plain}
\newtheorem{thm}{Theorem}[section]
\newtheorem{prop}[thm]{Proposition}
\newtheorem{lem}[thm]{Lemma}
\newtheorem{cor}[thm]{Corollary}
\newtheorem{remark}{Remark}[section]

\newtheorem{defn}{Definition}[section]

\numberwithin{equation}{section}

\title{Global Finite-Energy Weak Solutions and Sharp Entropy Decay
for a Poisson--Nernst--Planck System with Interspecies Drag and Steric Effects}
\author{Baoli Hao$^1$, Fanze Kong$^2$, Kei Fong Lam$^3$, Chun Liu$^1$} 
\date{ }

\begin{document}

\maketitle
\begin{center}
$^1$ Department of Applied Mathematics, Illinois Institute of Technology, 10 West 35th Street,
Chicago, IL 60616 \\ {\tt bhao2@hawk.illinoistech.edu, cliu124@illinoistech.edu}

\noindent $^2$ Department of Applied Mathematics, University of Washington, 4182 West Stevens Way Northeast, Seattle, WA 98195\\
{\tt fzkong@uw.edu}

\noindent $^3$ Department of Mathematics, Hong Kong Baptist University, 224 Waterloo Road, Kowloon Tong, Hong Kong \\
{\tt akflam@hkbu.edu.hk}
\end{center}

\renewcommand{\thefootnote}{\fnsymbol{footnote}}

\begin{abstract}
We derive and analyze a binary Poisson--Nernst--Planck system with steric interactions and interspecies drag through the energetic variational approach. The steric effects are incorporated into the free energy, while the drag mechanism enters the dissipation functional; eliminating the transport velocities yields a non-diagonal, concentration-dependent Onsager mobility and an entropy-production structure that is not coercive in the standard \(L^2(0,T;H^1)\) sense. For the resulting drag-modified steric PNP system, we prove the existence of global finite-energy weak solutions using an entropy-variable approximation, weighted gradient estimates, and a vacuum-compatible square-root formulation of the weighted entropy gradients. In the pure Neumann equal-mass setting, we establish a sublevel entropy-entropy production inequality, obtain exponential relaxation for approximation-generated weak solutions, and identify the sharp small-sublevel limit of the optimal entropy-production constant through an explicit linearized formula involving the drag mobility, steric Hessian, Poisson coupling, and Neumann spectrum. We further show that the same linearized constant governs the local nonlinear relaxation of sufficiently small strong perturbations of the homogeneous equilibrium. Finally, we discuss the rank-one steric limit and clarify the role of the positive definiteness of the steric matrix in the finite-energy compactness theory.
\end{abstract}



\section{Introduction}
\label{sec:introduction}


The Poisson--Nernst--Planck (PNP) system is a classical continuum model for
charged-particle transport driven by diffusion and electrostatic interaction.
It couples Nernst--Planck drift-diffusion equations for ionic concentrations to the Poisson equation for the electrostatic potential \cite{markowich2012semiconductor,jerome2012analysis}.  PNP-type models arise in semiconductor theory \cite{markowich2012semiconductor,jerome2012analysis},
electrochemistry \cite{rubinstein1990electro}, and biological ion channels \cite{eisenberg1998ionic,schuss2001derivation}. In its classical form, PNP treats ions as point-like species whose transport is governed by ideal mixing entropy and electrostatic interaction. It is therefore most appropriate for describing dilute regimes, where short-range particle interactions beyond electrostatics are negligible.

In concentrated electrolytes, narrow ion channels, and other crowded ionic environments, such short-range particle interactions can no longer be ignored. These steric effects are distinct from electrostatic attraction or repulsion:
they represent finite-size, excluded-volume, or short-range repulsive
interactions between ions, and are modeled here through an excess free-energy
contribution. They may affect ionic transport in two complementary ways. On the conservative
side, finite ion sizes generate steric repulsion and modify the chemical
potentials through excess free-energy contributions
\cite{borukhov1997steric,kilic2007steric,bazant2011double}. In ion-channel
models, such finite-size effects have also been represented through
Lennard--Jones type repulsive interactions and their local approximations
\cite{horng2012pnp,lin2014new}. On the dissipative side, relative motion
between different ionic species produces interspecies friction, a mechanism
closely related to multicomponent diffusion and Maxwell--Stefan type transport
\cite{taylor1993multicomponent,krishna2019diffusing,hsieh2015transport}.
Thus, crowded ionic transport involves both conservative particle interactions, which modify the free energy, and dissipative particle interactions that modify the dissipation mechanism, thereby altering the flux law.

Modified PNP models with steric or finite-size effects have been widely studied;
many of them incorporate conservative particle interactions through modified
free energies, chemical potentials, nonlinear mobilities, or volume-filling
effects
\cite{borukhov1997steric,kilic2007steric,bazant2011double,lin2013poisson,
horng2012pnp,lin2014new}. By contrast, interspecies drag acts at the level of
the dissipation and therefore changes the flux law rather than only the free
energy. The energetic variational approach (EnVarA), rooted in the Onsager
variational principle
\cite{onsager1931reciprocalI,onsager1931reciprocalII}, provides a natural
framework for combining these two mechanisms, since it starts from both a free
energy and a dissipation functional. In the present work, we use this framework
to derive and analyze a PNP-type system in which steric effects are incorporated into
the free energy, while interspecies drag is encoded in the dissipation and
leads to a non-diagonal, concentration-dependent Onsager mobility. The full
derivation is given in Section~\ref{sec:md}. We first recall the classical
variational structure in the nondimensional variables
used throughout the analysis, and then present the modified nondimensional system and its entropy-production identity.

\medskip

\noindent\textbf{Classical PNP Model. }
Let
\(
    \bc=(c_n,c_p)^\top
\)
denote the concentrations of the negatively and positively charged species.
The classical free energy is
\begin{equation}
E_{\rm cl}(\bc)
=
\int_\Omega
\left[
c_n\ln c_n+c_p\ln c_p
+\frac12|\nabla\phi|^2
\right]dx,
\qquad
-\Delta\phi=c_p-c_n .
\label{eq:intro-classical-free-energy}
\end{equation}
Its variational derivatives with respect to $c_n$ and $c_p$, denoted by $\mu_n^{\rm cl}$ and $\mu_p^{\rm cl}$, are given as
\[    \mu_n^{\rm cl}=\ln c_n+1-\phi,
    \qquad
    \mu_p^{\rm cl}=\ln c_p+1+\phi ,
\]
respectively. The classical dissipation contains only species-wise ionic friction,
\begin{equation}
    \mathcal D_{\rm cl}(\bu_n,\bu_p)
    =
    \int_\Omega
    \left(
    \frac1a c_n|\bu_n|^2
    +
    \frac1b c_p|\bu_p|^2
    \right)dx,
    \qquad a,b>0 .
\label{eq:intro-classical-velocity-dissipation}
\end{equation}
The corresponding force-balance relation is
\[
    \bu_n=-a\nabla\mu_n^{\rm cl},
    \qquad
    \bu_p=-b\nabla\mu_p^{\rm cl}.
\]
Together with the continuity equations, this yields the classical PNP system
\begin{subequations}
\begin{align}
    \partial_t\bc
    &=
    \nabla\cdot\bigl(M_{\rm cl}(\bc)\nabla\bmu^{\rm cl}\bigr),
    \\
    -\Delta\phi
    &=
    c_p-c_n ,
\end{align}
\label{eq:intro-classical-pnp}
\end{subequations}
with
\(
    \bmu^{\rm cl}=(\mu_n^{\rm cl},\mu_p^{\rm cl})^\top \) and concentration-dependent mobility matrix
    \(
    M_{\rm cl}(\bc)
    =
    \begin{pmatrix}
    a c_n & 0\\
    0 & b c_p
    \end{pmatrix}.
\)
For smooth positive solutions satisfying no-flux boundary conditions, the following energy identity holds
\[
    \frac{d}{dt}E_{\rm cl}(\bc(t))
    +
    \int_\Omega
    \left(
    a c_n|\nabla\mu_n^{\rm cl}|^2
    +
    b c_p|\nabla\mu_p^{\rm cl}|^2
    \right)dx
    =0 .
\]

\medskip

\noindent\textbf{Modified PNP model with enhanced drag and steric effects. }
The model studied in this paper modifies the preceding energy-dissipation
structure in two key aspects. First, steric interactions are incorporated into the free
energy:
\begin{equation}
E(\bc)
=
E_{\rm cl}(\bc)
+
\int_\Omega
\left[
\frac12 f_{nn}c_n^2
+f_{np}c_nc_p
+\frac12 f_{pp}c_p^2
\right]dx .
\label{eq:intro-free-energy}
\end{equation}
In the main finite-energy theory, the steric matrix
\(
    F=
    \begin{pmatrix}
    f_{nn} & f_{np}\\
    f_{np} & f_{pp}
    \end{pmatrix}
\)
is assumed to be symmetric positive definite. The corresponding
electrochemical potentials are
\[
\mu_n
=
\ln c_n+1-\phi+f_{nn}c_n+f_{np}c_p,
\qquad
\mu_p
=
\ln c_p+1+\phi+f_{np}c_n+f_{pp}c_p .
\]
Second, the dissipation is modified by adding relative drag between the two
ionic species:
\begin{equation}
    \mathcal D(\bu_n,\bu_p)
    =
    \int_\Omega
    \left[
    \frac1a c_n|\bu_n|^2
    +
    \frac1b c_p|\bu_p|^2
    +
    \frac1\delta c_nc_p|\bu_n-\bu_p|^2
    \right]dx,
    \qquad a,b,\delta>0 .
\label{eq:intro-velocity-dissipation}
\end{equation}
Eliminating the velocities from the force-balance relations gives the
non-diagonal Onsager mobility
\begin{equation}
M(\bc)
=
\frac{1}{\omega(\bc)}
\begin{pmatrix}
a c_n(\delta+b c_n)
&
ab c_nc_p
\\[1mm]
ab c_nc_p
&
b c_p(\delta+a c_p)
\end{pmatrix},
\qquad
\omega(\bc)=\delta+a c_p+b c_n .
\label{eq:intro-mobility}
\end{equation}
The resulting nondimensional system is
\begin{subequations}
\begin{align}
\partial_t \bc
&=
\nabla\cdot\bigl(M(\bc)\nabla\bmu\bigr),
\label{eq:intro-system-c}
\\
-\Delta\phi
&=
c_p-c_n .
\label{eq:intro-system-phi}
\end{align}
\label{eq:intro-system}
\end{subequations}
For smooth positive solutions satisfying no-flux boundary conditions, the
system satisfies
\begin{equation}
    \frac{d}{dt}E(\bc(t))
    +
    \mathcal D(\bc(t))
    =0 ,
\label{eq:intro-energy-dissipation}
\end{equation}
where
\begin{equation}
\mathcal D(\bc)
=
\int_\Omega
\nabla\bmu^\top M(\bc)\nabla\bmu\,dx
=
\int_\Omega
\frac{
\delta\bigl(a c_n|\nabla\mu_n|^2+b c_p|\nabla\mu_p|^2\bigr)
+
ab\bigl|c_n\nabla\mu_n+c_p\nabla\mu_p\bigr|^2
}
{\omega(\bc)}
\,dx .
\label{eq:intro-dissipation-decomposition}
\end{equation}
The identity \eqref{eq:intro-dissipation-decomposition} is the starting point
for the weak formulation and a priori estimates below. 

We now discuss the above energy-dissipation structure in the context of existing
analytical results for steric PNP and related cross-diffusion systems. Global
weak theories for steric PNP systems have first been developed for models in
which the cross-diffusion is generated by steric corrections to the chemical
potentials. For the two-species PNP system with local steric effects,
Hsieh~\cite{hsieh2019global} proved global existence by truncating the
diffusion matrix, applying Schauder's fixed-point theorem, and using the energy
inequality to obtain uniform \(L^2\)-bounds. This result is closely related to
the free-energy part of the present model, where in both cases a positive definite
steric interaction matrix provides coercive control of the concentrations. A
second line of work treats steric and ion-transport cross-diffusion systems
through entropy variables and the boundedness-by-entropy method
\cite{jungel2015boundedness,jungel2016entropy}. In this direction,
Gerstenmayer--J\"ungel~\cite{gerstenmayer2018analysis} analyzed a degenerate
cross-diffusion system arising from volume-filling effects, while
Hirvonen--J\"ungel~\cite{hirvonen2024analysis} studied a steric PNP
cross-diffusion system generated by localized Lennard--Jones interactions.
These works provide the closest analytical precedents for both the existence and long-time parts of our analysis. In those systems, however, the main coercive structure is tied directly to the entropy variables generated by the excess chemical potentials. In the present model, part of the coupling is carried instead by the drag-modified mobility \(M(\bc)\) in \eqref{eq:intro-mobility}, or equivalently by the collective dissipation mode in \eqref{eq:intro-dissipation-decomposition}. This distinction affects not only the compactness argument, but also the passage from entropy-entropy production estimates to finite-energy weak limits.

The interspecies drag mechanism itself is motivated by the modified PNP model
of Hsieh--Hyon--Lee--Lin--Liu~\cite{hsieh2015transport}, where an additional
entropy-production term depending on relative ionic velocities is introduced
through the maximum dissipation principle. That work provides the
dissipation-level modeling background for the drag term used here, but its
analysis establishes only local classical solutions by Galerkin's method and
Schauder's fixed-point theorem. Thus, it does not provide the global
finite-energy compactness theory needed for the present drag-modified steric
system. The main analytical task here is to combine the \(L^2\)-coercivity
supplied by the steric free energy with the weaker compactness information
available from the drag-modified entropy production. The first consequence is that the entropy production does not provide the
standard estimate
\(
    c_n,c_p\in L^2(0,T;H^1(\Omega)).
\)
The replacement is the weighted gradient control
\[
    \int_0^T\int_\Omega
    \frac{|\nabla c_n|^2+|\nabla c_p|^2}
         {1+c_n+c_p}
    \,dxdt
    \le C ,
\]
which is proved in Lemmas~\ref{lem:global-entropy-production}
and~\ref{lem:global-weighted-gradient}. Together with the
\(L^\infty(0,T;L^2(\Omega))\)-bound supplied by the positive definite steric
energy, this estimate is the compactness mechanism used in the finite-energy
existence theory. The second issue is the formulation of the weak limit. Entropy variables are a
standard tool in cross-diffusion analysis
\cite{jungel2015boundedness,jungel2016entropy}, and they are also used in the
approximation scheme below. However, finite-energy limits may contain vacuum,
while the electrochemical potentials exhibit a logarithmic dependence on the
concentrations. Thus, the weak formulation cannot be based directly on pointwise
entropy gradients \(\nabla\mu_n\) and \(\nabla\mu_p\). Instead, Definition~\ref{def:finite-energy-weak-solution} uses a square-root formulation of the weighted entropy gradients associated with \eqref{eq:intro-dissipation-decomposition}; the corresponding flux bounds are given in Lemma~\ref{lem:global-flux-estimate}. This formulation preserves the constitutive identification even in the presence of vacuum. The remaining issue in the long-time analysis is instead the relation between the intrinsic finite-energy dissipation and the lower-semicontinuous entropy production entering the sublevel inequality; see Remark~\ref{rem:dissipation-defect}. 
Finally, the positive definiteness of the steric matrix \(F\) is essential in
the main finite-energy theory, since it gives species-wise \(L^2\)-control of
\(c_n\) and \(c_p\). In the rank-one steric limit
\(
    F
    =
    f
    \begin{pmatrix}
    1&1\\
    1&1
    \end{pmatrix},
    \; f>0,
\)
the energy controls only the total density \(c_n+c_p\), while the charge mode
\(c_p-c_n\) still enters the Poisson equation and the individual fluxes. This
loss of species-wise control is the obstruction identified in
Proposition~\ref{prop:rank-one-intro}, and it is why the rank-one case is
treated separately as a degenerate limit.

\medskip

\noindent\textbf{Main Results. } We impose the following assumptions throughout the paper.

\begin{itemize}

\item[(A1)] \textbf{Domain.}
Let \(\Omega\subset\mathbb R^d\), \(d\le3\), be a bounded \(C^2\) domain.
For \(T>0\), we write
\(
    \Omega_T:=\Omega\times(0,T).
\)
Additional smoothness of \(\partial\Omega\) is assumed only in
Subsection~\ref{subsec:small-stability}, where strong solutions near the
homogeneous equilibrium are considered.

\item[(A2)] \textbf{Model parameters and steric matrix.}
The  mobility parameters satisfy
\(
    a>0,\; b>0,\; \delta>0,
\)
and we set
\(
    \omega(\bc):=\delta+a c_p+b c_n,
    \;
    \bc=(c_n,c_p)^\top,
    \;
    \bmu=(\mu_n,\mu_p)^\top.
\)
For the main finite-energy and long-time results, the steric interaction
matrix
\(
    F=
    \begin{pmatrix}
    f_{nn} & f_{np}\\
    f_{np} & f_{pp}
    \end{pmatrix}
\)
is symmetric positive definite. We denote its smallest eigenvalue by
\(\alpha_F>0\).

\item[(A3)] \textbf{Boundary conditions and Poisson realization.}
The ionic fluxes satisfy the no-flux boundary conditions
\(
    (M(\bc)\nabla\bmu)_i\cdot\nu=0,
    \; i=n,p
     \text{ on }\partial\Omega,
\)
with unit outward normal vector $\nu$. For the electrostatic potential, we fix a self-adjoint nonnegative Poisson
realization
\(
    \phi=\mathcal P\rho,
    \;
    \rho:=c_p-c_n,
    \;
    \rho\in\mathcal X_\Phi\subset L^2(\Omega),
\)
where \(\phi\) solves \(-\Delta\phi=\rho\) with homogeneous Dirichlet or Neumann boundary
conditions. In the homogeneous Dirichlet case we take \(\mathcal X_\Phi = L^2(\Omega)\), while in the homogeneous Neumann case we take \(\mathcal X_\Phi = L^2_0(\Omega) := \{ g \in L^2(\Omega) \, : \, \int_\Omega g \, dx = 0\}\) and supplement the additional requirement that \(\int_\Omega \mathcal P \rho \, dx = 0\). We assume
\[
    \|\mathcal P \rho\|_{H^2(\Omega)}
    +
    \|\nabla (\mathcal P \rho)\|_{L^6(\Omega)}
    \le
    C_\phi\|\rho\|_{L^2(\Omega)}
\]
holds with a positive constant $C_\phi$ independent of $\rho \in \mathcal X_\Phi$, and for all \(\rho,\eta\in\mathcal X_\Phi\),
\[
    \int_\Omega \nabla(\mathcal P\rho)\cdot\nabla(\mathcal P\eta)\,dx
    =
    \int_\Omega \rho\,\mathcal P\eta\,dx
    =
    \int_\Omega \eta\,\mathcal P\rho\,dx .
\]
In particular, \(\mathcal P\) is nonnegative. 

\item[(A4)] \textbf{Initial data.}
The initial concentrations satisfy
\(   c_{n,0},c_{p,0}\ge0,
    \;
    c_{n,0},c_{p,0}\in L^2(\Omega),
\)
and have finite entropy:
\[
    \int_\Omega c_{n,0}\ln c_{n,0}\,dx
    +
    \int_\Omega c_{p,0}\ln c_{p,0}\,dx
    <\infty .
\]
The initial charge
\(    \rho_0:=c_{p,0}-c_{n,0}
\)
belongs to the admissible charge space \(\mathcal X_\Phi\) in
\textup{(A3)}. In the homogeneous Neumann case this means
\(
    \int_\Omega c_{n,0}\,dx
    =
    \int_\Omega c_{p,0}\,dx .
\)

\end{itemize}
For the long-time behavior and sharp-rate results, we impose one additional
assumption.

\begin{itemize}

\item[(A5)] \textbf{Pure Neumann equal-mass setting.}
We specialize the Poisson realization in \textup{(A3)} to the homogeneous
Neumann case and assume that \(\Omega\) is connected. The conserved masses of
the two species are equal and positive:
\(
    \int_\Omega c_{n,0}\,dx
    =
    \int_\Omega c_{p,0}\,dx
    =
    m|\Omega|,
    \;
    m>0 .
\)
Under the no-flux boundary conditions, these masses remain conserved in time.
The corresponding homogeneous equilibrium is given by the pair
\(
    \bc^\infty=(m,m)^\top,
    \;
    \phi^\infty=0 .
\)

\end{itemize}

The rank-one steric case is not included in \textup{(A2)} and is treated
separately as a degenerate limit. There the positive definiteness condition is
replaced by
\(
    F=
    f
    \begin{pmatrix}
    1&1\\
    1&1
    \end{pmatrix},
    \; f>0.
\)

Our first main result is the existence of global finite-energy weak solutions. The
formulation uses a vacuum-compatible square-root representation of the weighted entropy gradients rather than the singular quantities $\nabla \mu_i$ themselves, which is essential because finite-energy solutions may contain
vacuum.

\begin{thm}[Global finite-energy weak solutions]
\label{thm:global-weak-intro}
Assume \textup{(A1)}--\textup{(A4)}. Then, for every \(T>0\), there exists a
finite-energy weak solution \((c_n,c_p,\phi)\) on  $(0,T)$ in the sense of
Definition~\ref{def:finite-energy-weak-solution}. In particular,
\[
c_n,c_p\ge0,
\qquad
c_n,c_p\in
L^\infty(0,T;L^2(\Omega))
\cap
L^{4/3}(0,T;W^{1,4/3}(\Omega)),
\]
and
\[
\partial_t c_n,\partial_t c_p
\in
L^{4/3}(0,T;W^{1,4}(\Omega)').
\]
The Poisson equation
\(
-\Delta\phi=c_p-c_n
\)
holds in the weak sense, with the boundary conditions specified in
\textup{(A3)}. 
Moreover, the weighted entropy-gradient fields \(\boldsymbol A_n,\boldsymbol A_p\) defined by \eqref{eq:weak-weighted-entropy-gradients}, together with the collective dissipation field \(\boldsymbol B\) defined by \eqref{eq:weak-collective-field}, satisfy \( \boldsymbol A_n,\boldsymbol A_p,\boldsymbol B \in L^2(\Omega_T;\mathbb R^d). \) The corresponding physical fluxes \(\boldsymbol J_n,\boldsymbol J_p\) in \eqref{eq:weak-physical-fluxes} belong to \(L^{4/3}(\Omega_T;\mathbb R^d)\).
The initial data are attained
in \(W^{1,4}(\Omega)'\), and the energy inequality
\[
E(\bc(t))
+
\int_0^t\int_\Omega
\bigl(
|\boldsymbol A_n|^2
+
|\boldsymbol A_p|^2
+
|\boldsymbol B|^2
\bigr)
\,dxds
\le
E(\bc_0)
\]
holds for a.e. \(t\in(0,T)\).
\end{thm}

The proof is based on an entropy-variable implicit Euler approximation. The positive definite steric energy and the weighted gradient estimate extracted from the entropy production provide compactness of the concentrations. The square-root formulation then identifies the limiting weighted entropy gradients, including on vacuum sets, and allows passage to the limit in the nonlinear physical fluxes.

The second part of the paper concerns long-time behavior in the pure Neumann
equal-mass setting \textup{(A5)}. In this case the homogeneous state
\(
    \bc^\infty=(m,m)^\top,
    \;
    \phi^\infty=0
\)
is the equilibrium. We prove an entropy-entropy production inequality on
bounded energy sublevels and use it to obtain exponential relaxation for the
finite-energy weak solutions generated by the mass-preserving approximation.


\begin{thm}[Sublevel entropy-entropy production and decay]
\label{thm:eep-decay-intro}
Assume \textup{(A1)}--\textup{(A5)}. For every \(E_0>E(\bc^\infty)\), there
exists a constant \(\lambda(E_0)>0\) such that every admissible positive state
\(\bc\) satisfying the mass constraints and
\(
    E(\bc)\le E_0
\)
obeys
\[
    \mathcal D(\bc)
    \ge
    \lambda(E_0)\bigl(E(\bc)-E(\bc^\infty)\bigr).
\]
Consequently, if \((\bc,\phi)\) is a finite-energy weak solution obtained as a
limit of the mass-preserving entropy approximation, then, for every
\(E_*>E(\bc_0)\),
\[
    E(\bc(t))-E(\bc^\infty)
    \le
    e^{-\lambda(E_*)t}
    \bigl(E(\bc_0)-E(\bc^\infty)\bigr)
\]
for a.e. \(t>0\), where \(\lambda(E_*)\) is the sublevel
entropy-entropy production constant corresponding to the energy level
\(E_*\).
\end{thm}

The approximation qualification in Theorem~\ref{thm:eep-decay-intro} is not due to a loss of constitutive information. Rather, the sublevel entropy-entropy production inequality is proved for positive states and first applied along the mass-preserving entropy approximation. Its extension to general finite-energy weak solutions requires a no-defect relation between the finite-energy dissipation associated with \(\boldsymbol A_n,\boldsymbol A_p,\boldsymbol B\) and the lower-semicontinuous extension of \(\mathcal D\); see Remark~\ref{rem:dissipation-defect}.

The next result identifies the sharp small-sublevel limit of the optimal
constants in the entropy-entropy production inequality. Let
\(\{\nu_k\}_{k\ge1}\) be the nonzero Neumann eigenvalues of \(-\Delta\),
let
\(
    M_0:=M(\bc^\infty) \) be the Onsager mobility associated to $\bc^\infty = (m,m)^{\top}$, where \(m = \frac{1}{|\Omega|} \int_\Omega c_{n,0} \, dx = \frac{1}{|\Omega|} \int_{\Omega} c_{p,0} \, dx\),
    \(
    H_0:=m^{-1}I+F\),
    and \(
    \bz:=(-1,1)^\top .
\)
For \(E_0>E(\bc^\infty)\), define
\(
    \Lambda(E_0)
    :=
    \inf
    \left\{
    \frac{\mathcal D(\bc)}
    {E(\bc)-E(\bc^\infty)}
    :
    \begin{array}{l}
    \bc \ \text{admissible},\;
    E(\bc)\le E_0, \;
    \bc\ne \bc^\infty
    \end{array}
    \right\}.
\)

\begin{thm}[Sharp small-sublevel limit]
\label{thm:sharp-constant-intro}
Under the assumptions of Theorem~\ref{thm:eep-decay-intro},
\[
    \lim_{E_0\downarrow E(\bc^\infty)}
    \Lambda(E_0)
    =
    \lambda_{\rm lin},
\]
where
\[
    \lambda_{\rm lin}
    =
    2\inf_{k\ge1} \left (
    \nu_k\,
    \lambda_{\min}
    \left(
    M_0
    \left(
    H_0+\nu_k^{-1}\bz\otimes\bz
    \right)
    \right) \right ).
\]
Here \(\lambda_{\min}(M_0A)\) denotes the smallest eigenvalue of the
symmetric matrix \(A^{1/2}M_0A^{1/2}\), with
\(A=H_0+\nu_k^{-1}\bz\otimes\bz\).
\end{thm}

Thus, the long-time theory is not only qualitative: near equilibrium, the optimal entropy-production constant converges to an explicit linearized rate determined by the drag mobility, the steric Hessian, the Poisson coupling, and the Neumann spectrum. As a smooth dynamical consequence, the same constant governs the local nonlinear relaxation of sufficiently small strong perturbations of the homogeneous equilibrium; see Subsection~\ref{subsec:small-stability}.

Finally, we discuss the rank-one steric limit. This case is not covered by the
main finite-energy existence theorem, but it clarifies the role of the positive
definiteness assumption on \(F\).

\begin{prop}[Rank-one steric limit]
\label{prop:rank-one-intro}
Assume that the positive definiteness condition in \textup{(A2)} is replaced by
\(
    F=
    f
    \begin{pmatrix}
    1&1\\
    1&1
    \end{pmatrix},
    \; f>0.
\)
Let
\(
    u=c_n+c_p,
    \;
    \rho=c_p-c_n.
\)
For smooth positive solutions, or for smooth entropy approximations satisfying
uniform energy and dissipation bounds, the total density \(u\) satisfies the following compactness assertion:
\[
    u^k\to u
    \text{ strongly in }L^2(0,T;L^2(\Omega)),
\]
up to a subsequence. However, the available estimates do not provide the same
compactness for the charge mode \(\rho\).
\end{prop}

In the rank-one case, the collective drag mode controls the total density
\(u=c_n+c_p\), but not the charge mode \(\rho=c_p-c_n\). Since \(\rho\) still
enters both the Poisson equation and the individual fluxes, this missing
charge-mode compactness prevents a direct extension of the full finite-energy
weak theory.

The rest of the paper is organized as follows. In Section~\ref{sec:md}, we
derive the drag-modified steric PNP system from the energetic variational
approach and carry out the nondimensionalization leading to
\eqref{eq:intro-system}. Section~\ref{sec:global-weak} is devoted to the proof
of the global finite-energy weak existence theorem. We first set up the
vacuum-compatible square-root weak formulation and the basic a priori estimates, then construct
solutions by an entropy-variable approximation and pass to the limit in the
nonlinear fluxes. The rank-one steric limit is discussed at the end of
Section~\ref{sec:global-weak}. Section~\ref{sec:long-time} establishes the sublevel entropy-entropy production inequality in the pure Neumann equal-mass setting, derives exponential relaxation for approximation-generated weak solutions, discusses the extension to general finite-energy weak solutions, identifies the sharp small-sublevel entropy-production constant, and concludes with its nonlinear stability consequence near equilibrium.


\section{Model derivation}
\label{sec:md}

\subsection{Energetic variational derivation}
\label{subsec:derivation}

We derive the drag-modified steric PNP system from the energetic variational
approach (EnVarA)
\cite{eisenberg2010energy,hyon2012energy,hsieh2015transport}. We start from a
Helmholtz free energy with ideal entropy, electrostatic energy, and local
steric interaction, and from a velocity-level dissipation functional with
solvent-ion friction and interspecies drag. The force-balance relations
obtained from these two variational principles determine the Onsager mobility
used in the analysis below.

Let \(\Omega\subset\mathbb R^d\) be a bounded domain, and let \(c_n\) and
\(c_p\) denote the concentrations of the negatively and positively charged
ionic species. The two species carry charges \(-z_0e\) and \(+z_0e\), where
\(z_0>0\) is the valence number and \(e\) is the elementary charge. Denoting
their transport velocities by \(\boldsymbol u_n\) and \(\boldsymbol u_p\), the
conservation laws are
\begin{equation}
    \partial_t c_n+\nabla\cdot(c_n\boldsymbol u_n)=0,
    \qquad
    \partial_t c_p+\nabla\cdot(c_p\boldsymbol u_p)=0.
\label{eq:dim-continuity}
\end{equation}
The electric potential \(\phi\) is determined by
\begin{equation}
    -\epsilon\Delta\phi=z_0e(c_p-c_n),
\label{eq:dim-poisson}
\end{equation}
where \(\epsilon>0\) is the dielectric permittivity.

It is useful to make explicit that the electrostatic energy is a nonlocal
interaction between the two charged species. Let \(\mathcal G_\Omega\) denote
the Green operator for \(-\Delta\) with the prescribed boundary conditions.
Then \eqref{eq:dim-poisson} can be written as
\(
    \phi
    =
    \frac{z_0e}{\epsilon}\,
    \mathcal G_\Omega(c_p-c_n).
\)
Consequently,
\begin{equation}
    \frac{\epsilon}{2}\int_\Omega |\nabla\phi|^2\,dx
    =
    \frac{(z_0e)^2}{2\epsilon}
    \int_\Omega
    (c_p-c_n)\,\mathcal G_\Omega(c_p-c_n)\,dx .
\label{eq:electrostatic-nonlocal}
\end{equation}
Equivalently, if \(\mathsf G_\Omega(x,y)\) is the Green kernel associated with
\(\mathcal G_\Omega\), then
\[
    \frac{\epsilon}{2}\int_\Omega |\nabla\phi|^2\,dx
    =
    \frac{(z_0e)^2}{2\epsilon}
    \int_\Omega\int_\Omega
    \mathsf G_\Omega(x,y)
    (c_p-c_n)(x)(c_p-c_n)(y)\,dxdy .
\]
In the pure Neumann case, the Green operator is defined  on the zero-mean
charge space, together with the normalization \(\int_\Omega\phi\,dx=0\).

We model steric effects by a local quadratic excess free-energy density, \(
    \frac12 g_{nn}c_n^2+g_{np}c_nc_p+\frac12 g_{pp}c_p^2,
\)
as in local PNP-steric approximations
\cite{horng2012pnp,lin2014new,hsieh2019global}. Here the coefficients
\(g_{ij}\) describe effective short-range repulsion between ionic species. We
assume that \(
    G:=
    \begin{pmatrix}
    g_{nn} & g_{np}\\
    g_{np} & g_{pp}
    \end{pmatrix}
\)
is symmetric positive definite. 
The Helmholtz free energy is then
\begin{equation}
\begin{aligned}
\mathcal F(c_n,c_p)
=
\int_\Omega
\bigg[
\underbrace{
K_B\theta
\left(
c_n\ln\frac{c_n}{c_{\rm ref}}
+
c_p\ln\frac{c_p}{c_{\rm ref}}
\right)
}_{\text{Ideal mixing entropy}}
+
\underbrace{
\frac{\epsilon}{2}|\nabla\phi|^2
}_{\text{Electrostatic energy}}
+
\underbrace{
\frac12 g_{nn}c_n^2
+g_{np}c_nc_p
+\frac12 g_{pp}c_p^2
}_{\text{Steric interaction energy}}
\bigg]\,dx,
\end{aligned}
\label{eq:dim-free-energy}
\end{equation}
where \(K_B\) is the Boltzmann constant, \(\theta\) is the absolute
temperature, and \(c_{\rm ref}\) is a reference concentration.


The nonlocal representation \eqref{eq:electrostatic-nonlocal} gives
\[
    \delta
    \left(
    \frac{\epsilon}{2}\int_\Omega |\nabla\phi|^2\,dx
    \right)
    =
    z_0e\int_\Omega
    \phi\,\delta(c_p-c_n)\,dx .
\]
Therefore the electrochemical potentials are
\begin{equation}
\begin{aligned}
\mu_n
:=
\frac{\delta\mathcal F}{\delta c_n}
&=
K_B\theta
\left(
\ln\frac{c_n}{c_{\rm ref}}+1
\right)
-z_0e\phi
+g_{nn}c_n+g_{np}c_p,
\\
\mu_p
:=
\frac{\delta\mathcal F}{\delta c_p}
&=
K_B\theta
\left(
\ln\frac{c_p}{c_{\rm ref}}+1
\right)
+z_0e\phi
+g_{np}c_n+g_{pp}c_p.
\end{aligned}
\label{eq:dim-chemical-potentials}
\end{equation}
The terms \(-z_0e\phi\) and \(+z_0e\phi\) are the electrostatic contributions
to the electrochemical potentials, while the linear concentration terms arise
from the local steric interaction.

We compute the conservative force densities by varying the free energy along
mass-preserving transport variations. For each species \(i\in\{n,p\}\), let
\(\boldsymbol h_i\) be an Eulerian virtual displacement such that \(\boldsymbol h_i \cdot \boldsymbol \nu = 0\) on \( \partial \Omega\) and consider the
perturbed map
\(
    T_i^\varepsilon(x)
    =
    x+\varepsilon\boldsymbol h_i(x).
\)
Let \(c_i^\varepsilon\) be the density obtained by transporting \(c_i\) through
\(T_i^\varepsilon\). The conservation of mass yields the following pointwise mass-preserving relation after a change of variables
\[  c_i^\varepsilon(T_i^\varepsilon(x))
    \det\nabla T_i^\varepsilon(x)
    =
    c_i(x).
\]
Differentiating at \(\varepsilon=0\) gives
\begin{equation}
    \delta c_i
    :=
    \left.\frac{d}{d\varepsilon}\right|_{\varepsilon=0}
    c_i^\varepsilon
    =
    -\nabla\cdot(c_i\boldsymbol h_i),
    \qquad i\in\{n,p\}.
\label{eq:dim-concentration-variation}
\end{equation}
Using \eqref{eq:dim-chemical-potentials}, and varying one species at a time,
we obtain
\begin{equation}
\begin{aligned}
\delta_i\mathcal F
=
\int_\Omega \mu_i\,\delta c_i\,dx
=
-\int_\Omega
\mu_i\,\nabla\cdot(c_i\boldsymbol h_i)\,dx
=
\int_\Omega
c_i\nabla\mu_i\cdot\boldsymbol h_i\,dx,
\qquad i\in\{n,p\}.
\end{aligned}
\label{eq:dim-energy-transport-variation}
\end{equation}
Thus the least action principle gives the conservative force densities
\begin{equation}
    \boldsymbol F_i^{\rm con}
    =
    -c_i\nabla\mu_i,
    \qquad i\in\{n,p\}.
\label{eq:dim-conservative-forces}
\end{equation}

The velocity-level dissipation contains the standard solvent-ion friction and
an additional relative-drag term between the two species. We take
\begin{equation}
\begin{aligned}
\mathcal D
=
\int_\Omega
\bigg[
\frac{K_B\theta}{D_n}c_n|\boldsymbol u_n|^2
+
\frac{K_B\theta}{D_p}c_p|\boldsymbol u_p|^2
+
\frac{K_B\theta}{D_{n,p}}
c_nc_p|\boldsymbol u_n-\boldsymbol u_p|^2
\bigg]\,dx,
\end{aligned}
\label{eq:dim-dissipation}
\end{equation}
where \(D_n,D_p>0\) are the single-species diffusion coefficients and
\(D_{n,p}>0\) is the relative-drag parameter
\cite{taylor1993multicomponent,krishna2019diffusing,hsieh2015transport}. The
factor \(c_nc_p\) reflects that the drag is generated by interactions between
the two ionic species and vanishes when either species is absent.

By the maximum dissipation principle, the dissipative force densities are
obtained by varying \(\frac12\mathcal D\) with respect to the velocities:
\begin{equation}
\begin{aligned}
\boldsymbol F_n^{\rm dis}
:=
\frac{\delta(\frac12\mathcal D)}
     {\delta\boldsymbol u_n}
&=
\frac{K_B\theta}{D_n}c_n\boldsymbol u_n
+
\frac{K_B\theta}{D_{n,p}}
c_nc_p(\boldsymbol u_n-\boldsymbol u_p),
\\
\boldsymbol F_p^{\rm dis}
:=
\frac{\delta(\frac12\mathcal D)}
     {\delta\boldsymbol u_p}
&=
\frac{K_B\theta}{D_p}c_p\boldsymbol u_p
+
\frac{K_B\theta}{D_{n,p}}
c_nc_p(\boldsymbol u_p-\boldsymbol u_n).
\end{aligned}
\label{eq:dim-dissipative-forces}
\end{equation}
The force-balance law
\[
    \boldsymbol F_i^{\rm dis}=\boldsymbol F_i^{\rm con},
    \qquad i\in\{n,p\},
\]
therefore yields
\begin{equation}
\begin{aligned}
\frac{K_B\theta}{D_n}c_n\boldsymbol u_n
+
\frac{K_B\theta}{D_{n,p}}
c_nc_p(\boldsymbol u_n-\boldsymbol u_p)
&=
-c_n\nabla\mu_n,
\\
\frac{K_B\theta}{D_p}c_p\boldsymbol u_p
+
\frac{K_B\theta}{D_{n,p}}
c_nc_p(\boldsymbol u_p-\boldsymbol u_n)
&=
-c_p\nabla\mu_p.
\end{aligned}
\label{eq:dim-force-balance}
\end{equation}

Introducing the ionic fluxes
\(
    \boldsymbol J_n:=c_n\boldsymbol u_n,
    \;
    \boldsymbol J_p:=c_p\boldsymbol u_p,
\)
and solving \eqref{eq:dim-force-balance}, we obtain
\begin{equation}
\begin{pmatrix}
\boldsymbol J_n\\[1mm]
\boldsymbol J_p
\end{pmatrix}
=
-
\mathcal M(c_n,c_p)
\begin{pmatrix}
\nabla\mu_n\\[1mm]
\nabla\mu_p
\end{pmatrix},
\label{eq:dim-mobility-form}
\end{equation}
where
\(
\mathcal M(c_n,c_p)
=
\frac{1}{K_B\theta\,\Xi(c_n,c_p)}
\begin{pmatrix}
D_nc_n(D_{n,p}+D_pc_n)
&
D_nD_pc_nc_p
\\[1mm]
D_nD_pc_nc_p
&
D_pc_p(D_{n,p}+D_nc_p)
\end{pmatrix},
\label{eq:dim-mobility-matrix}
\)
with
\(
    \Xi(c_n,c_p)
    :=
    D_{n,p}+D_nc_p+D_pc_n .
\)
The matrix \(\mathcal M(c_n,c_p)\) is symmetric and positive semidefinite for
nonnegative concentrations. Combining \eqref{eq:dim-mobility-form} with
\eqref{eq:dim-continuity} and \eqref{eq:dim-poisson} gives the
drag-modified steric PNP system.

If the relative-drag term in \eqref{eq:dim-dissipation} is absent, the
mobility reduces to the diagonal solvent-ion mobility of the classical PNP
model. Thus, the off-diagonal entries in \eqref{eq:dim-mobility-matrix} are
precisely the contribution of interspecies drag.

\subsection{Nondimensionalization and energy dissipation}
\label{subsec:nondimensionalization}

We nondimensionalize the mobility formulation derived above. Let \(D_*>0\) be
a reference diffusivity and define the Debye length
\(
    \lambda_D
    :=
    \left(
    \frac{\epsilon K_B\theta}{z_0^2e^2c_{\rm ref}}
    \right)^{1/2}.
\)
We introduce
\[
    x=\lambda_D\widehat x,
    \quad
    t=\frac{\lambda_D^2}{D_*}\widehat t,
    \quad
    c_i=c_{\rm ref}\widehat c_i,
    \quad
    \phi=\frac{K_B\theta}{z_0e}\widehat\phi,
\quad
    \mu_i=K_B\theta\,\widehat\mu_i,
    \quad
    \boldsymbol J_i
    =
    \frac{D_*c_{\rm ref}}{\lambda_D}\widehat{\boldsymbol J}_i,
    \quad i\in\{n,p\}.
\]
The dimensionless free energy \(E\) is defined by
\(
    \mathcal F
    =
    K_B\theta\,c_{\rm ref}\lambda_D^d E .
\)

We also set
\[
    f_{ij}:=\frac{c_{\rm ref}}{K_B\theta}g_{ij},
    \quad
    a:=\frac{D_n}{D_*},
    \quad
    b:=\frac{D_p}{D_*},
    \quad
    \delta:=\frac{D_{n,p}}{D_*c_{\rm ref}} .
\]
Thus, the dimensionless steric matrix
\(
    F=
    \begin{pmatrix}
    f_{nn} & f_{np}\\
    f_{np} & f_{pp}
    \end{pmatrix}
\)
is symmetric positive definite whenever the dimensional steric matrix \(G\) is
symmetric positive definite.

After dropping hats and keeping the same notation for the rescaled variables and domain, the dimensional formulation derived above becomes the nondimensional system stated in \eqref{eq:intro-system}. In particular, the rescaled free energy, mobility matrix, and entropy-production decomposition are given by \eqref{eq:intro-free-energy}, \eqref{eq:intro-mobility}, and \eqref{eq:intro-dissipation-decomposition}, respectively. These are the forms used throughout the rest of the paper.


\section{Global finite-energy weak solutions}
\label{sec:global-weak}
In this section, we prove Theorem~\ref{thm:global-weak-intro} and discuss the rank-one steric case in Proposition~\ref{prop:rank-one-intro}. The main difficulty is the degeneracy encoded in the entropy production \eqref{eq:intro-dissipation-decomposition}: finite-energy solutions do not provide species-wise Fisher-information bounds, and vacuum may prevent \(\nabla\ln c_i\), hence \(\nabla\mu_i\), from being meaningful as functions. For this reason, after collecting the necessary properties of the Poisson operator, the free energy, and the entropy-variable map in Section~\ref{subsec:global-preparations}, we formulate the problem in Definition~\ref{def:finite-energy-weak-solution} using a square-root formulation of the weighted entropy gradients naturally controlled by \eqref{eq:intro-dissipation-decomposition}. The a priori estimates in Section~\ref{subsec:apriori-estimates} show how the positive definite steric energy combines with the weighted gradient estimate to give the compactness needed for the nonlinear fluxes. We then construct approximate solutions by the entropy-variable scheme in Section~\ref{subsec:entropy-variable-approximation}, obtain uniform bounds and compactness in Section~\ref{subsec:uniform-bounds-compactness}, and pass to the limit in Section~\ref{subsec:passage-to-limit}. Finally, Section~\ref{subsec:rank-one-steric} explains why the same argument no longer applies in the rank-one steric case, where the energy controls only the total density and not the charge mode.

\subsection{Preliminaries}
\label{subsec:global-preparations}

We collect several basic facts that will be used repeatedly in the existence
proof. Throughout this section, we keep the notation introduced in
Section~\ref{sec:introduction}. In particular,
\[
    \bc=(c_n,c_p)^\top,\qquad
    \bmu=(\mu_n,\mu_p)^\top,\qquad
    \rho=c_p-c_n,\qquad
    \omega(\bc)=\delta+a c_p+b c_n .
\]
The purpose of this subsection is to isolate the analytical ingredients needed
below: the Poisson operator, the coercivity and convexity of the free energy,
and the invertibility of the entropy-variable map.

We first fix the operator notation for the Poisson equation. This allows us to
treat the electrostatic part of the free energy as a convex quadratic form in
the charge density.

\begin{lem}[Poisson operator and electrostatic energy]
\label{lem:poisson-operator}
Let \(\rho\in L^2(\Omega)\) be an admissible charge density, and let
\(
    \phi=\mathcal P\rho
\)
denote the solution of
\(
    -\Delta\phi=\rho
\)
with the boundary conditions specified in \textup{(A3)}. In the pure Neumann
case, admissibility means additionally
\(
    \int_\Omega \rho\,dx=0,
\)
and the potential is normalized by
\(
    \int_\Omega \phi\,dx=0.
\)
Then, by \textup{(A3)},
\[
    \|\mathcal P\rho\|_{H^2(\Omega)}
    +
    \|\nabla\mathcal P\rho\|_{L^6(\Omega)}
    \le
    C_\phi\|\rho\|_{L^2(\Omega)}.
\]
Moreover, \(\mathcal P\) is self-adjoint and nonnegative on the admissible
charge space, and
\[
    \frac12\int_\Omega |\nabla\phi|^2\,dx
    =
    \frac12\langle \rho,\mathcal P\rho\rangle .
\]
Consequently, the electrostatic energy is a nonnegative convex quadratic
functional of the charge density \(\rho\).
\end{lem}

The next estimate establishes the key  species-wise \(L^2\)-control. The positive definiteness of the steric matrix is essential for its proof.

\begin{lem}[Coercivity of the free energy]
\label{lem:energy-coercivity}
There exists a constant \(C>0\), depending only on $\Omega$, such that every admissible nonnegative state
\(\bc=(c_n,c_p)^\top\)  of (\ref{eq:intro-system-c})-(\ref{eq:intro-system-phi}) satisfies
\[
    E(\bc)
    \ge
    \frac{\alpha_F}{2}
    \int_\Omega
    \bigl(c_n^2+c_p^2\bigr)\,dx
    -
    C ,
\]
where $E$ is given in (\ref{eq:intro-free-energy}) and $\alpha_F$ is the constant specified in \textup{(A2)}.  Consequently, bounded energy sublevels are bounded in \(L^2(\Omega)^2\).
\end{lem}

\begin{proof}
For every \(\xi\in\mathbb R^2\),
\(
    \xi^\top F\xi\ge \alpha_F|\xi|^2 .
\)
Therefore, the steric part of the energy satisfies
\(
    \frac12
    \int_\Omega \bc^\top F\bc\,dx
    \ge
    \frac{\alpha_F}{2}
    \int_\Omega |\bc|^2\,dx .
\)
The electrostatic contribution is nonnegative in its original form
\(\frac12\int_\Omega|\nabla\phi|^2\,dx\). Finally,
\(s\ln s\ge -e^{-1}\) for all \(s\ge0\). Combining these estimates gives the
claim.
\end{proof}

Besides coercivity, we shall use the convexity of the free energy in the
discrete entropy estimate. The electrostatic contribution is convex since it
is a nonnegative quadratic form in the charge density, possibly plus fixed
affine and constant terms, and the charge density depends linearly on \(\bc\).

\begin{lem}[Convexity inequality]
\label{lem:global-convexity}
Let \(\bc=(c_n,c_p)^\top\) and
\(\widetilde{\bc}=(\widetilde c_n,\widetilde c_p)^\top\) be admissible
nonnegative states with finite energy. Assume that \(\bc\) is strictly
positive and sufficiently smooth so that the corresponding electrochemical
potential \(\bmu(\bc)\) is well-defined and the right-hand side below is
finite. Then
\begin{align}\label{state_estimate}
    E(\bc)-E(\widetilde{\bc})
    \le
    \int_\Omega
    \bmu(\bc)\cdot(\bc-\widetilde{\bc})\,dx.
\end{align}
\end{lem}

\begin{proof}
The functions \(s\mapsto s\ln s\) are convex on \([0,\infty)\), and the
steric energy is convex since \(F\) is symmetric positive definite. By using
Lemma~\ref{lem:poisson-operator}, the electrostatic energy is convex in
\(c_p-c_n\). Since \(c_p-c_n\) depends linearly on \(\bc\), the electrostatic
energy is convex as a function of \(\bc\). Therefore, \(E\) is convex on the
admissible set, and estimate (\ref{state_estimate}) is the subgradient inequality
evaluated at \(\bc\). The subgradient is precisely the electrochemical
potential \(\bmu(\bc)\).
\end{proof}

The approximation scheme will be formulated in entropy variables. In the
present notation, the entropy variables are precisely the electrochemical
potentials \(\bmu=(\mu_n,\mu_p)^\top\). The inverse map from entropy variables
to concentrations takes values in the positive cone, where the logarithmic
entropy is well-defined.

\begin{lem}[Invertibility of entropy variables]
\label{lem:global-entropy-invertibility}
For every \(\bmu=(\mu_n,\mu_p)^\top\in\mathbb R^2\) and every
\(\phi\in\mathbb R\), there exists a unique
\(\bc(\bmu,\phi)\in(0,\infty)^2\) such that
\[
    \mu_n
    =
    \ln c_n+1-\phi+f_{nn}c_n+f_{np}c_p,
\qquad
    \mu_p
    =
    \ln c_p+1+\phi+f_{np}c_n+f_{pp}c_p.
\]
Moreover, the map \((\bmu,\phi)\mapsto \bc(\bmu,\phi)\) is continuous.
\end{lem}

\begin{proof}
For fixed \((\bmu,\phi)\), consider the function
\[
\begin{aligned}
h(\bc)
=
c_n(\ln c_n-1)+c_p(\ln c_p-1)
+\frac12\bc^\top F\bc
+
(1-\phi-\mu_n)c_n
+
(1+\phi-\mu_p)c_p
\end{aligned}
\]
on \((0,\infty)^2\). Its Hessian is
\(
    D^2h(\bc)
    =
    \begin{pmatrix}
        1/c_n & 0\\
        0 & 1/c_p
    \end{pmatrix}
    +
    F,
\)
which is positive definite. Hence \(h\) is strictly convex. 
We extend \(h\) continuously to the closed cone \([0,\infty)^2\), using
\(c\ln c=0\) at \(c=0\). Since \(F\) is positive definite, the quadratic term dominates the linear
terms, and hence \(h(\bc)\to+\infty\) as \(|\bc|\to\infty\). Thus \(h\)
attains a minimum on \([0,\infty)^2\). We claim that this minimum cannot occur
on the boundary. Indeed, if a minimizer had \(c_n=0\), then for fixed \(c_p\)
and small \(\varepsilon>0\),
\[
h(\varepsilon,c_p)-h(0,c_p)
=
\varepsilon\ln\varepsilon
+
\varepsilon
\bigl(f_{np}c_p-\phi-\mu_n\bigr)
+
\frac12 f_{nn}\varepsilon^2 .
\]
Since \(\varepsilon\ln\varepsilon+C\varepsilon+O(\varepsilon^2)<0\) for all
sufficiently small \(\varepsilon>0\), we have
\(h(\varepsilon,c_p)<h(0,c_p)\), which is a contradiction. The same argument excludes the possibility of 
\(c_p=0\). Therefore the minimizer lies in \((0,\infty)^2\). The Euler--Lagrange equations
of this minimization problem are exactly the entropy-variable relations above.
Continuity follows from the implicit function theorem, since
\(D^2h(\bc)\) is positive definite.
\end{proof}

\medskip
\noindent\textbf{Finite-energy weak formulation.}  We now define the weak formulation used in the proof of
Theorem~\ref{thm:global-weak-intro}. The entropy-variable formulation is
natural at the approximation level, where the concentrations are positive.
For finite-energy limits, however, vacuum may occur, and the quantities
\(\nabla\ln c_i\), hence \(\nabla\mu_i\), are not suitable as primitive weak
objects. We therefore formulate the limiting equations in terms of a
square-root formulation of the weighted entropy gradients associated with
\eqref{eq:intro-dissipation-decomposition}.

\begin{defn}[Finite-energy weak solution]
\label{def:finite-energy-weak-solution}
Assume \textup{(A1)}--\textup{(A4)} and let \(T>0\). A triple
\(
    (c_n,c_p,\phi)
\)
is called a finite-energy weak solution on \((0,T)\) if
\[
    c_n,c_p\ge0,
    \qquad
    c_n,c_p
    \in
    L^\infty(0,T;L^2(\Omega))
    \cap
    L^{4/3}(0,T;W^{1,4/3}(\Omega)),
\]
\[
    \partial_t c_n,\partial_t c_p
    \in
    L^{4/3}\bigl(0,T;W^{1,4}(\Omega)'\bigr),
\]
and
\(
    \phi=\mathcal P(c_p-c_n)
\)
with the boundary conditions specified in \textup{(A3)}. Define
\(
    \chi_n
    :=
    \sqrt{\frac{\delta a c_n}{\omega(\bc)}},
    \;
    \chi_p
    :=
    \sqrt{\frac{\delta b c_p}{\omega(\bc)}}.
\)
We require
\(
    \chi_n,\chi_p
    \in
    L^{4/3}\bigl(0,T;W^{1,4/3}(\Omega)\bigr).
\)
With all gradients understood in the weak sense, define
\begin{equation}
\begin{aligned}
    \boldsymbol A_n
    &:=
    2\nabla\chi_n
    +
    \chi_n
    \nabla
    \bigl(
        \log\omega(\bc)
        -\phi
        +f_{nn}c_n
        +f_{np}c_p
    \bigr),
    \\
    \boldsymbol A_p
    &:=
    2\nabla\chi_p
    +
    \chi_p
    \nabla
    \bigl(
        \log\omega(\bc)
        +\phi
        +f_{np}c_n
        +f_{pp}c_p
    \bigr),
\end{aligned}
\label{eq:weak-weighted-entropy-gradients}
\end{equation}
and require
\(
    \boldsymbol A_n,\boldsymbol A_p
    \in L^2(\Omega_T;\mathbb R^d).
\)
The collective dissipation field is defined by
\begin{equation}
    \boldsymbol B
    :=
    \sqrt{\frac{b c_n}{\delta}}\,
    \boldsymbol A_n
    +
    \sqrt{\frac{a c_p}{\delta}}\,
    \boldsymbol A_p,
\label{eq:weak-collective-field}
\end{equation}
and is required to satisfy
\(
    \boldsymbol B
    \in L^2(\Omega_T;\mathbb R^d).
\)

The physical fluxes are defined by
\begin{equation}
\begin{aligned}
    \boldsymbol J_n
    :=
    -\chi_n\boldsymbol A_n
    -
    c_n\sqrt{\frac{ab}{\omega(\bc)}}\,
    \boldsymbol B,
    \quad
    \boldsymbol J_p
    :=
    -\chi_p\boldsymbol A_p
    -
    c_p\sqrt{\frac{ab}{\omega(\bc)}}\,
    \boldsymbol B.
\end{aligned}
\label{eq:weak-physical-fluxes}
\end{equation}
Then
\(
    \boldsymbol J_n,\boldsymbol J_p
    \in L^{4/3}(\Omega_T;\mathbb R^d),
\)
and, for every
\(
    \varphi_i\in L^4(0,T;W^{1,4}(\Omega)),
    \; i=n,p,
\)
one has
\[
    \int_0^T
    \langle\partial_t c_i,\varphi_i\rangle\,dt
    -
    \int_{\Omega_T}
    \boldsymbol J_i\cdot\nabla\varphi_i\,dxdt
    =
    0.
\]
The initial data are attained in the sense that
\[
    c_i(t)\to c_{i,0}
    \quad\text{in }W^{1,4}(\Omega)'
    \quad\text{as }t\downarrow0,
    \qquad i=n,p.
\]
Finally, for a.e. \(t\in(0,T)\),
\[
    E(\bc(t))
    +
    \int_0^t\int_\Omega
    \left(
        |\boldsymbol A_n|^2
        +
        |\boldsymbol A_p|^2
        +
        |\boldsymbol B|^2
    \right)\,dxds
    \le
    E(\bc_0).
\]
\end{defn}

\begin{remark}[Consistency and vacuum compatibility]
\label{rem:weak-dissipation-fields}
For smooth strictly positive solutions, the chain rule gives
\(
    2\nabla\chi_n
    =
    \chi_n
    \bigl(
        \nabla\log c_n
        -
        \nabla\log\omega(\bc)
    \bigr), \)
and \(
    2\nabla\chi_p
    =
    \chi_p
    \bigl(
        \nabla\log c_p
        -
        \nabla\log\omega(\bc)
    \bigr).
\)
Hence, by \eqref{eq:weak-weighted-entropy-gradients},
\(
    \boldsymbol A_n
    =
    \sqrt{\frac{\delta a c_n}{\omega(\bc)}}\,
    \nabla\mu_n,
    \;
    \boldsymbol A_p
    =
    \sqrt{\frac{\delta b c_p}{\omega(\bc)}}\,
    \nabla\mu_p.
\)
It follows that \(\boldsymbol B\) and the fluxes in
\eqref{eq:weak-collective-field}--\eqref{eq:weak-physical-fluxes} reduce to
their smooth constitutive expressions,
\[
    \boldsymbol J=-M(\bc)\nabla\bmu,
\]
and
\[
    \mathcal D(\bc)
    =
    \int_\Omega
    \left(
        |\boldsymbol A_n|^2
        +
        |\boldsymbol A_p|^2
        +
        |\boldsymbol B|^2
    \right)\,dx.
\]

The formulation remains meaningful on vacuum sets. Indeed, since
\(    \{\chi_n=0\}=\{c_n=0\},
    \;
    \{\chi_p=0\}=\{c_p=0\},
\)
the Sobolev gradients satisfy
\(
    \nabla\chi_i=0
    \;\;\text{a.e. on }\{c_i=0\}.
\)
Therefore,
\(
    \boldsymbol A_i=0
    \;\;\text{a.e. on }\{c_i=0\},
    \text{ for } i=n,p.
\)
Thus, the fields
\(\boldsymbol A_n,\boldsymbol A_p,\boldsymbol B\) are intrinsically determined
by \((c_n,c_p,\phi)\), rather than being additional independent solution
variables.
\end{remark}

\subsection{A priori estimates}
\label{subsec:apriori-estimates}

The following estimates form the compactness mechanism of the existence
proof. The entropy production gives a weighted gradient bound for the
concentrations; combined with the \(L^\infty(0,T;L^2(\Omega))\)-control from
Lemma~\ref{lem:energy-coercivity}, this yields the Sobolev regularity and flux
bounds required in Definition~\ref{def:finite-energy-weak-solution}.

We first extract from the entropy production the weighted gradient control
that replaces the unavailable species-wise \(H^1\)-bound.

\begin{lem}[Entropy-production estimate]
\label{lem:global-entropy-production}
Let \(\bc=(c_n,c_p)^\top\) be a smooth strictly positive admissible state, and
let \(\phi=\mathcal P(c_p-c_n)\). Then there exist constants
\(\kappa_{\rm w}>0\) and \(C>0\), depending only on the fixed data of the problem,
such that
\[
\mathcal D(\bc)
\ge
\kappa_{\rm w}
\int_\Omega
\frac{
|\nabla c_n|^2+|\nabla c_p|^2
}
{1+c_n+c_p}
\,dx
-
C\bigl(1+E(\bc)\bigr)^{3/2}.
\]
\end{lem}

\begin{proof}
Set
\(
    H(\bc)
    :=
    \begin{pmatrix}
        1/c_n & 0\\
        0 & 1/c_p
    \end{pmatrix}
    +
    F,
    \;
    \bz:=(-1,1)^\top .
\)
Then
\(    \nabla\bmu
    =
    H(\bc)\nabla\bc+\bz\,\nabla\phi .
\)
All pointwise matrix inequalities below are interpreted in \(\mathbb R^2\) and
then summed over the spatial directions.

We first establish the algebraic coercivity of the mobility combined with the
entropy Hessian. From \eqref{eq:intro-mobility}, for \(c_n,c_p>0\), it holds that
\[
M(\bc)^{-1}
=
\begin{pmatrix}
\dfrac{\delta+a c_p}{a\delta c_n}
&
-\dfrac1\delta
\\[1mm]
-\dfrac1\delta
&
\dfrac{\delta+b c_n}{b\delta c_p}
\end{pmatrix}.
\]

Since \(F\) is positive definite, a direct computation from the explicit forms of \(H(\bc)\) and \(M(\bc)\) gives 
\[ H(\bc)^{-1}M(\bc)^{-1}H(\bc)^{-1} \le C(1+c_n+c_p)I , \] with \(C=C(a,b,\delta,F)\).
Equivalently, after decreasing \(\kappa_{\rm w}>0\) if necessary,
\[
    H(\bc)M(\bc)H(\bc)
    \ge
    \frac{\kappa_{\rm w}}{1+c_n+c_p}I.
\]
Using Young's inequality in the \(M(\bc)\)-inner product, we obtain
\[
\begin{aligned}
\nabla\bmu^\top M(\bc)\nabla\bmu
&=
\bigl(H(\bc)\nabla\bc+\bz\nabla\phi\bigr)^\top
M(\bc)
\bigl(H(\bc)\nabla\bc+\bz\nabla\phi\bigr)
\\
&\ge
\frac12
\nabla\bc^\top H(\bc)M(\bc)H(\bc)\nabla\bc
-
(\bz\nabla\phi)^\top M(\bc)(\bz\nabla\phi)
\\
&\ge
\kappa_{\rm w}
\frac{
|\nabla c_n|^2+|\nabla c_p|^2
}
{1+c_n+c_p}
-
C(1+c_n+c_p)|\nabla\phi|^2 ,
\end{aligned}
\]
where we have used
\[
    \bz^\top M(\bc)\bz
    =
    \frac{
    \delta(a c_n+b c_p)+ab(c_n-c_p)^2
    }
    {\omega(\bc)}
    \le
    C(1+c_n+c_p).
\]
Integrating over \(\Omega\) gives
\[
\mathcal D(\bc)
\ge
\kappa_{\rm w}
\int_\Omega
\frac{
|\nabla c_n|^2+|\nabla c_p|^2
}
{1+c_n+c_p}
\,dx
-
C
\int_\Omega
(1+c_n+c_p)|\nabla\phi|^2\,dx .
\]

It remains to estimate the electrostatic term. By using \textup{(A3)}, one finds
\[
    \|\nabla\phi\|_{L^6(\Omega)}
    \le
    C_\phi\|c_p-c_n\|_{L^2(\Omega)}.
\]
Using H\"older's inequality, the boundedness of \(\Omega\), and \(d\le3\), we have
\[
\begin{aligned}
\int_\Omega(1+c_n+c_p)|\nabla\phi|^2\,dx
&\le
\|\nabla\phi\|_{L^2(\Omega)}^2
+
\|c_n+c_p\|_{L^{3/2}(\Omega)}
\|\nabla\phi\|_{L^6(\Omega)}^2
\\
&\le
C
\bigl(1+\|c_n\|_{L^2(\Omega)}+\|c_p\|_{L^2(\Omega)}\bigr)
\|c_p-c_n\|_{L^2(\Omega)}^2 .
\end{aligned}
\]
By Lemma~\ref{lem:energy-coercivity}, one finds
\[
    \|c_n\|_{L^2(\Omega)}^2+\|c_p\|_{L^2(\Omega)}^2
    \le
    C\bigl(1+E(\bc)\bigr).
\]
Hence
\[
    \int_\Omega(1+c_n+c_p)|\nabla\phi|^2\,dx
    \le
    C\bigl(1+E(\bc)\bigr)^{3/2}.
\]
The claim follows.
\end{proof}

This weighted estimate is weaker than an \(L^2(0,T;H^1(\Omega))\)-bound, but
it is sufficient to obtain the Sobolev regularity used in the weak
formulation.

\begin{lem}[Weighted gradient estimate]
\label{lem:global-weighted-gradient}
Assume that
\(
    c_n,c_p\ge0,
    \;
    c_n,c_p\in L^\infty(0,T;L^2(\Omega)),
\)
and
\(
    \int_0^T\int_\Omega
    \frac{|\nabla c_n|^2+|\nabla c_p|^2}
    {1+c_n+c_p}
    \,dxdt
    <\infty .
\)
Then
\(
    c_n,c_p\in L^{4/3}(0,T;W^{1,4/3}(\Omega)).
\)
\end{lem}

\begin{proof}
For \(i=n,p\),
\[
|\nabla c_i|^{4/3}
=
\left(
\frac{|\nabla c_i|^2}{1+c_n+c_p}
\right)^{2/3}
(1+c_n+c_p)^{2/3}.
\]
Applying H\"older's inequality with exponents \(3/2\) and \(3\), we get
\[
\begin{aligned}
\int_0^T\int_\Omega |\nabla c_i|^{4/3}\,dxdt
&\le
\left(
\int_0^T\int_\Omega
\frac{|\nabla c_i|^2}{1+c_n+c_p}
\,dxdt
\right)^{2/3}
\left(
\int_0^T\int_\Omega
(1+c_n+c_p)^2
\,dxdt
\right)^{1/3}.
\end{aligned}
\]
The second factor is finite because
\(c_n,c_p\in L^\infty(0,T;L^2(\Omega))\). Together with
\(c_n,c_p\in L^\infty(0,T;L^2(\Omega))\), this gives the desired
\(L^{4/3}(0,T;W^{1,4/3}(\Omega))\)-bound.
\end{proof}

The same \(L^\infty(0,T;L^2(\Omega))\)-control also  guarantees sufficient integrability of the  physical fluxes so that the continuity equations are well defined in the weak formulation.

\begin{lem}[Flux estimate]
\label{lem:global-flux-estimate}
Assume that
\(    c_n,c_p\ge0,
    \;
    c_n,c_p\in L^\infty(0,T;L^2(\Omega)),
\)
and let
\(
    \boldsymbol A_n,\boldsymbol A_p,\boldsymbol B
    \in L^2(\Omega_T;\mathbb R^d).
\)
Define the physical fluxes by
\(
    \boldsymbol J_n
    =
    -
    \sqrt{\frac{\delta a c_n}{\omega(\bc)}}\,\boldsymbol A_n
    -
    c_n\sqrt{\frac{ab}{\omega(\bc)}}\,\boldsymbol B,
\)
and
\(
    \boldsymbol J_p
    =
    -
    \sqrt{\frac{\delta b c_p}{\omega(\bc)}}\,\boldsymbol A_p
    -
    c_p\sqrt{\frac{ab}{\omega(\bc)}}\,\boldsymbol B.
\)
Then
\(
    \boldsymbol J_n,\boldsymbol J_p
    \in L^{4/3}(\Omega_T;\mathbb R^d).
\)
More precisely,
\[
\begin{aligned}
\|\boldsymbol J_n\|_{L^{4/3}(\Omega_T)}
+
\|\boldsymbol J_p\|_{L^{4/3}(\Omega_T)}
 \le
C
\bigl(
1+\|c_n\|_{L^\infty(0,T;L^2)}^{1/2}
&+\|c_p\|_{L^\infty(0,T;L^2)}^{1/2}
\bigr)
\\
&
\bigl(
\|\boldsymbol A_n\|_{L^2(\Omega_T)}
+
\|\boldsymbol A_p\|_{L^2(\Omega_T)}
+
\|\boldsymbol B\|_{L^2(\Omega_T)}
\bigr).
\end{aligned}
\]
\end{lem}

\begin{proof}
It is sufficient to estimate the two terms in each flux. Since
\(\omega(\bc)\ge\delta\), it holds
\(
    \frac{\delta a c_n}{\omega(\bc)}
    \le
    a c_n,
    \;
    \frac{\delta b c_p}{\omega(\bc)}
    \le
    b c_p .
\)
Moreover, since
\(\omega(\bc)=\delta+a c_p+b c_n\), we have
\(
    \frac{ab c_n^2}{\omega(\bc)}
    \le
    C c_n,
    \;
    \frac{ab c_p^2}{\omega(\bc)}
    \le
    C c_p .
\)
Thus, each coefficient in the flux representation is bounded by a constant
multiple of \(\sqrt{c_n}\) or \(\sqrt{c_p}\). For example,
\[
\left\|
\sqrt{\frac{\delta a c_n}{\omega(\bc)}}\,\boldsymbol A_n
\right\|_{L^{4/3}(\Omega_T)}
\le
C
\|\sqrt{c_n}\|_{L^4(\Omega_T)}
\|\boldsymbol A_n\|_{L^2(\Omega_T)}.
\]
Since
\[
    \|\sqrt{c_n}\|_{L^4(\Omega_T)}
    =
    \left(
    \int_0^T\int_\Omega c_n^2\,dxdt
    \right)^{1/4}
    \le
    T^{1/4}
    \|c_n\|_{L^\infty(0,T;L^2(\Omega))}^{1/2},
\]
this term belongs to \(L^{4/3}(\Omega_T)\). The remaining terms are treated in
the same way. This proves the estimate.
\end{proof}

These estimates will be applied to the entropy-variable approximations
constructed in the next subsection.


\subsection{Entropy-variable approximation}
\label{subsec:entropy-variable-approximation}

We now construct approximate solutions by an implicit Euler scheme in entropy
variables. The scheme is based on the operator form of
\eqref{eq:intro-system}, with the physical flux convention
\(\boldsymbol J=-M(\bc)\nabla\bmu\). By
Lemma~\ref{lem:global-entropy-invertibility}, the inverse entropy-variable map
\((\bmu,\phi)\mapsto\bc(\bmu,\phi)\) takes values in the positive cone; hence
the logarithmic entropy is well-defined at the approximate level.

Let \(\ell>d/2+1\) be an integer. For scalar functions
\(\xi,\eta\in H^\ell(\Omega)\), define
\begin{align}\label{bilinearform_b}
    b(\xi,\eta)
    :=
    \sum_{|\beta|\le \ell}
    \int_\Omega D^\beta\xi\,D^\beta\eta\,dx.
\end{align}
For vector-valued functions
\(\bxi=(\xi_n,\xi_p)^\top\) and
\(\boldsymbol\eta=(\eta_n,\eta_p)^\top\), we set
\[
    b(\bxi,\boldsymbol\eta)
    :=
    b(\xi_n,\eta_n)+b(\xi_p,\eta_p).
\]
Since \(\ell>d/2+1\), we have the continuous embedding
\(
    H^\ell(\Omega)\hookrightarrow W^{1,\infty}(\Omega),
\)
and the compact embedding
\(
    H^\ell(\Omega)\hookrightarrow L^\infty(\Omega).
\)

Let \(\tau>0\) and \(\varepsilon>0\). Given an admissible nonnegative state
\(
    \bc^{k-1}=(c_n^{k-1},c_p^{k-1})^\top,
\)
we seek
\[
    \bmu^k=(\mu_n^k,\mu_p^k)^\top\in H^\ell(\Omega)^2,
    \qquad
    \phi^k,
    \qquad
    \bc^k=\bc(\bmu^k,\phi^k)\in(0,\infty)^2,
\]
such that
\(
    -\Delta\phi^k=c_p^k-c_n^k
\)
with the boundary conditions specified in \textup{(A3)}, and
\begin{equation}
\frac1\tau
\int_\Omega
(c_i^k-c_i^{k-1})\eta_i\,dx
+
\int_\Omega
\bigl(M(\bc^k)\nabla\bmu^k\bigr)_i\cdot\nabla\eta_i\,dx
+
\varepsilon b(\mu_i^k,\eta_i)
=
0
\label{eq:global-regularized-discrete}
\end{equation}
for all \(\eta_i\in H^\ell(\Omega)\), \(i=n,p\).

In the pure Neumann case, the construction is performed in the mass-constrained setting. We write
\(
    \mu_i^k=\widetilde\mu_i^k+\zeta_i^k,
    \;
    \int_\Omega \widetilde\mu_i^k\,dx=0,
\)
and impose \eqref{eq:global-regularized-discrete} on the zero-mean test
space. The constants \(\zeta_i^k\) are chosen so that the species masses are
preserved:
\begin{equation}
    \int_\Omega c_i^k\,dx
    =
    \int_\Omega c_i^{k-1}\,dx,
    \qquad i=n,p .
\label{eq:global-discrete-mass-conservation}
\end{equation}
Consequently,
\[
    \int_\Omega(c_p^k-c_n^k)\,dx
    =
    \int_\Omega(c_p^{k-1}-c_n^{k-1})\,dx,
\]
so the compatibility condition for the Neumann Poisson problem is preserved by the evolution, provided it holds initially. The potential is normalized by
\(
    \int_\Omega\phi^k\,dx=0.
\)
In the entropy estimates below, the constant parts of the entropy variables
do not contribute because of the mass constraints
\eqref{eq:global-discrete-mass-conservation}.

\begin{lem}[Solvability of the regularized problem]
\label{lem:global-regularized-solvability}
For every \(\tau,\varepsilon>0\) and every admissible
\(\bc^{k-1}\), the regularized problem
\eqref{eq:global-regularized-discrete}, with the mass-constrained
interpretation above in the pure Neumann case, admits a solution.
\end{lem}

\begin{proof}
We use Leray--Schauder's fixed point theorem; see, for instance,
\cite{zeidler2013nonlinear}. The proof is written in the unconstrained
notation. In the pure Neumann case, the same argument is applied to the
zero-mean parts of the entropy variables, with the constants determined by
the mass constraints.

Let
\(
    \by=(y_n,y_p)^\top\in L^\infty(\Omega)^2,
    \;
    \theta\in[0,1].
\)
For fixed \(\by\), we first determine the electrostatic potential and the
corresponding concentrations. By
Lemma~\ref{lem:global-entropy-invertibility}, for each value of \(\phi\) there
is a unique
\(
    \bc=\bc(\by,\phi)\in(0,\infty)^2
\)
solving the entropy-variable relation
\[
    \by
    =
    \ln\bc+\mathbf 1+\bz\phi+F\bc,
    \qquad
    \bz:=(-1,1)^\top .
\]
We then solve the semilinear Poisson equation
\begin{equation}
    -\Delta\phi
    =
    c_p(\by,\phi)-c_n(\by,\phi)
\label{eq:global-semilinear-poisson}
\end{equation}
with the boundary conditions in \textup{(A3)}.

The solvability of \eqref{eq:global-semilinear-poisson} follows from
monotonicity. Differentiating the entropy-variable relation with respect to
\(\phi\), we obtain
\[
    \partial_\phi\bc
    =
    -
    \left(
    \operatorname{diag}(1/c_n,1/c_p)+F
    \right)^{-1}\bz .
\]
Therefore
\[
\begin{aligned}
    \partial_\phi(c_p-c_n)
    =
    \bz^\top\partial_\phi\bc
   =
    -
    \bz^\top
    \left(
    \operatorname{diag}(1/c_n,1/c_p)+F
    \right)^{-1}
    \bz
    \le 0 .
\end{aligned}
\]
Thus the operator
\[
    \phi\mapsto
    -\Delta\phi-\bigl(c_p(\by,\phi)-c_n(\by,\phi)\bigr)
\]
is monotone. The standard monotone-operator argument gives a solution
\(\phi=\phi(\by)\) of \eqref{eq:global-semilinear-poisson}; see, for example,
\cite{brezis2011functional}. We set
\(
    \bc(\by):=\bc(\by,\phi(\by)).
\)

In the pure Neumann case, the same construction is performed together with
the mass constraints. More precisely, the additive constants in the entropy
variables are chosen so that
\[
    \int_\Omega c_i(\by,\phi)\,dx
    =
    \int_\Omega c_i^{k-1}\,dx,
    \qquad i=n,p .
\]
This gives the zero-mean compatibility condition for the right-hand side of
\eqref{eq:global-semilinear-poisson}, and the potential is normalized by zero
mean. The strict monotonicity of the entropy-variable map with respect to the
additive constants gives the required mass adjustment.

Next, for fixed \((\by,\theta)\), define
\(
    \bv=S(\by,\theta)\in H^\ell(\Omega)^2
\)
as the solution of the linear regularized problem
\begin{equation}
\theta
\int_\Omega
M(\bc(\by))\nabla\bv:\nabla\boldsymbol\eta\,dx
+
\varepsilon b(\bv,\boldsymbol\eta)
=
-\frac{\theta}{\tau}
\int_\Omega
\bigl(\bc(\by)-\bc^{k-1}\bigr)\cdot\boldsymbol\eta\,dx
\label{eq:global-fixed-point-map}
\end{equation}
for all
\(
    \boldsymbol\eta\in H^\ell(\Omega)^2 .
\)
In the pure Neumann case, \eqref{eq:global-fixed-point-map} is imposed on the
zero-mean test space. Since \(M(\bc(\by))\) is positive semidefinite and
\(\varepsilon b\) is coercive on \(H^\ell(\Omega)^2\), the bilinear form on
the left-hand side is coercive. Hence, the Lax--Milgram theorem gives a unique
solution \(\bv=S(\by,\theta)\).

We now verify that \(S\) is compact and continuous as a map into
\(L^\infty(\Omega)^2\). Let \(\by\) remain in a bounded subset of
\(L^\infty(\Omega)^2\). The construction above gives a corresponding bound
for \(\bc(\by)\) in \(L^2(\Omega)^2\), depending on the \(L^\infty\)-bound for
\(\by\), the previous state \(\bc^{k-1}\), and the fixed data of the problem.
Testing \eqref{eq:global-fixed-point-map} with
\(\boldsymbol\eta=\bv\), we obtain
\[
    \theta
    \int_\Omega
    M(\bc(\by))\nabla\bv:\nabla\bv\,dx
    +
    \varepsilon b(\bv,\bv)
    =
    -\frac{\theta}{\tau}
    \int_\Omega
    \bigl(\bc(\by)-\bc^{k-1}\bigr)\cdot\bv\,dx .
\]
Dropping the nonnegative mobility term and using Cauchy's inequality gives
\[
    \varepsilon b(\bv,\bv)
    \le
    \frac{1}{\tau}
    \|\bc(\by)-\bc^{k-1}\|_{L^2(\Omega)}
    \|\bv\|_{L^2(\Omega)} ,
\]
where $b$ is given in (\ref{bilinearform_b}).
Since \(b(\bv,\bv)\) controls
\(
    \|\bv\|_{H^\ell(\Omega)^2}^2,
\)
it follows that
\[
    \|\bv\|_{H^\ell(\Omega)^2}
    \le
    C(\tau,\varepsilon,\bc^{k-1},\|\by\|_{L^\infty}).
\]
The compact embedding
\(
    H^\ell(\Omega)\hookrightarrow L^\infty(\Omega)
\)
shows that \(S\) maps bounded subsets of \(L^\infty(\Omega)^2\) into
relatively compact subsets of \(L^\infty(\Omega)^2\).

The continuity of \(S\) follows from the continuity of the entropy-variable
map, the continuous dependence in the monotone Poisson problem
\eqref{eq:global-semilinear-poisson}, and the continuous dependence of the
solution of \eqref{eq:global-fixed-point-map} on its coefficients and
right-hand side. Thus, \(S\) is a compact continuous fixed point map.

It remains to obtain an a priori bound for all fixed points of \(S\). Let
\(
    \bmu=S(\bmu,\theta),
    \;
    \bc=\bc(\bmu,\phi).
\)
Then \((\bmu,\bc,\phi)\) satisfies
\[
\frac{\theta}{\tau}
\int_\Omega
(\bc-\bc^{k-1})\cdot\boldsymbol\eta\,dx
+
\theta
\int_\Omega
M(\bc)\nabla\bmu:\nabla\boldsymbol\eta\,dx
+
\varepsilon b(\bmu,\boldsymbol\eta)
=
0
\]
for all \(\boldsymbol\eta\in H^\ell(\Omega)^2\). In the pure Neumann case,
we test with the zero-mean part of \(\bmu\); the constant part drops out of
the first term because of the mass constraints
\eqref{eq:global-discrete-mass-conservation}, and it has zero gradient in the
second term.

Taking \(\boldsymbol\eta=\bmu\), we get
\[
\frac{\theta}{\tau}
\int_\Omega
(\bc-\bc^{k-1})\cdot\bmu\,dx
+
\theta
\int_\Omega
\nabla\bmu^\top M(\bc)\nabla\bmu\,dx
+
\varepsilon b(\bmu,\bmu)
=
0 .
\]
By the definition of \(\mathcal D\),
\[
    \int_\Omega
    \nabla\bmu^\top M(\bc)\nabla\bmu\,dx
    =
    \mathcal D(\bc).
\]
Using Lemma~\ref{lem:global-convexity},
\[
    E(\bc)-E(\bc^{k-1})
    \le
    \int_\Omega
    \bmu\cdot(\bc-\bc^{k-1})\,dx .
\]
Therefore,
\[
    \frac{\theta}{\tau}
    \bigl(E(\bc)-E(\bc^{k-1})\bigr)
    +
    \theta\mathcal D(\bc)
    +
    \varepsilon b(\bmu,\bmu)
    \le
    0 .
\]
Since \(\mathcal D(\bc)\ge0\) and \(E\) is bounded from below on admissible
states,
\[
    \varepsilon b(\bmu,\bmu)
    \le
    \frac1\tau
    \bigl(E(\bc^{k-1})-\inf E\bigr).
\]
Thus, all fixed points are bounded in \(H^\ell(\Omega)^2\), uniformly for
\(\theta\in[0,1]\). Leray--Schauder's theorem yields a fixed point for
\(\theta=1\), and this fixed point solves
\eqref{eq:global-regularized-discrete}.
\end{proof}

\begin{lem}[Discrete entropy inequality]
\label{lem:global-discrete-entropy}
Every solution of the regularized problem
\eqref{eq:global-regularized-discrete}, with the mass-constrained
interpretation above in the pure Neumann case, satisfies
\[
    E(\bc^k)
    +
    \tau\mathcal D(\bc^k)
    +
    \varepsilon\tau b(\bmu^k,\bmu^k)
    \le
    E(\bc^{k-1}).
\]
\end{lem}

\begin{proof}
Choose
\(
    \eta_i=\mu_i^k,
    \; i=n,p,
\)
in \eqref{eq:global-regularized-discrete}, and sum over the two species. In
the pure Neumann constrained formulation, we test with the zero-mean parts of
\(\mu_i^k\). The constant parts do not contribute to the time-discrete term
because of \eqref{eq:global-discrete-mass-conservation}, and they have zero
gradient in the mobility term. Hence
\[
\frac1\tau
\int_\Omega
(\bc^k-\bc^{k-1})\cdot\bmu^k\,dx
+
\int_\Omega
\nabla\bmu^k{}^\top M(\bc^k)\nabla\bmu^k\,dx
+
\varepsilon b(\bmu^k,\bmu^k)
=
0 .
\]
The middle term is precisely \(\mathcal D(\bc^k)\). Therefore
\[
\frac1\tau
\int_\Omega
(\bc^k-\bc^{k-1})\cdot\bmu^k\,dx
+
\mathcal D(\bc^k)
+
\varepsilon b(\bmu^k,\bmu^k)
=
0 .
\]
By Lemma~\ref{lem:global-convexity},
\[
    E(\bc^k)-E(\bc^{k-1})
    \le
    \int_\Omega
    \bmu^k\cdot(\bc^k-\bc^{k-1})\,dx .
\]
Combining the last two relations gives
\[
    \frac1\tau
    \bigl(E(\bc^k)-E(\bc^{k-1})\bigr)
    +
    \mathcal D(\bc^k)
    +
    \varepsilon b(\bmu^k,\bmu^k)
    \le
    0 .
\]
Multiplying by \(\tau\) yields the desired inequality.
\end{proof}


\subsection{Uniform bounds and compactness}
\label{subsec:uniform-bounds-compactness}

We now pass from the discrete entropy inequality to estimates that are uniform
in the approximation parameters. These estimates are written in terms of the dissipation fields introduced in
Definition~\ref{def:finite-energy-weak-solution}, but now at the approximate
level.

Let \(N\tau=T\). For \(t\in((k-1)\tau,k\tau]\), define the piecewise constant
interpolants
\[
    \bc_{\tau,\varepsilon}(t)=\bc^k,
    \qquad
    \bmu_{\tau,\varepsilon}(t)=\bmu^k,
    \qquad
    \phi_{\tau,\varepsilon}(t)=\phi^k .
\]
We also define the left time-shift and the discrete time derivative by
\[
    \sigma_\tau\bc_{\tau,\varepsilon}(t)=\bc^{k-1},
    \qquad
    \partial_t^\tau\bc_{\tau,\varepsilon}(t)
    =
    \frac{\bc^k-\bc^{k-1}}{\tau},
    \qquad
    t\in((k-1)\tau,k\tau].
\]

For the approximate solutions, set
\(
    \chi_{n,\tau,\varepsilon}
    :=
    \sqrt{
        \frac{\delta a c_{n,\tau,\varepsilon}}
        {\omega(\bc_{\tau,\varepsilon})}
    }
\)
and
\(
    \chi_{p,\tau,\varepsilon}
    :=
    \sqrt{
        \frac{\delta b c_{p,\tau,\varepsilon}}
        {\omega(\bc_{\tau,\varepsilon})}
    }.
\)
We define
\(
    \boldsymbol A_{n,\tau,\varepsilon}
    :=
    \chi_{n,\tau,\varepsilon}
    \nabla\mu_{n,\tau,\varepsilon}
\)
and
\(
    \boldsymbol A_{p,\tau,\varepsilon}
    :=
    \chi_{p,\tau,\varepsilon}
    \nabla\mu_{p,\tau,\varepsilon}.
\)
At each discrete time step,
\(\bmu^k\in H^\ell(\Omega)^2\hookrightarrow W^{1,\infty}(\Omega)^2\),
and the entropy-variable reconstruction gives continuous strictly positive
concentrations on \(\overline\Omega\). Hence each \(c_i^k\) has a positive
lower bound at the fixed approximation level, and the Sobolev chain rule is
applicable. It follows that, a.e. in \(\Omega_T\),
\begin{equation}
\begin{aligned}
    \boldsymbol A_{n,\tau,\varepsilon}
    &=
    2\nabla\chi_{n,\tau,\varepsilon}
    +
    \chi_{n,\tau,\varepsilon}
    \nabla\Bigl(
        \log\omega(\bc_{\tau,\varepsilon})
        -\phi_{\tau,\varepsilon}
        +f_{nn}c_{n,\tau,\varepsilon}
        +f_{np}c_{p,\tau,\varepsilon}
    \Bigr),
    \\
    \boldsymbol A_{p,\tau,\varepsilon}
    &=
    2\nabla\chi_{p,\tau,\varepsilon}
    +
    \chi_{p,\tau,\varepsilon}
    \nabla\Bigl(
        \log\omega(\bc_{\tau,\varepsilon})
        +\phi_{\tau,\varepsilon}
        +f_{np}c_{n,\tau,\varepsilon}
        +f_{pp}c_{p,\tau,\varepsilon}
    \Bigr).
\end{aligned}
\label{eq:approximate-square-root-identities}
\end{equation}
We further define
\(
    \boldsymbol B_{\tau,\varepsilon}
    :=
    \sqrt{
        \frac{ab}{\omega(\bc_{\tau,\varepsilon})}
    }
    \bigl(
        c_{n,\tau,\varepsilon}\nabla\mu_{n,\tau,\varepsilon}
        +
        c_{p,\tau,\varepsilon}\nabla\mu_{p,\tau,\varepsilon}
    \bigr).
\)
By the definitions above, the pointwise compatibility identity
\begin{equation}
    \boldsymbol B_{\tau,\varepsilon}
    =
    \sqrt{
        \frac{b c_{n,\tau,\varepsilon}}{\delta}
    }\,
    \boldsymbol A_{n,\tau,\varepsilon}
    +
    \sqrt{
        \frac{a c_{p,\tau,\varepsilon}}{\delta}
    }\,
    \boldsymbol A_{p,\tau,\varepsilon}
\label{eq:approximate-compatibility}
\end{equation}
holds a.e. in \(\Omega_T\).


Consistently with the physical flux convention
\(\boldsymbol J=-M(\bc)\nabla\bmu\), define
\[
\boldsymbol J_{n,\tau,\varepsilon}
=
-
\sqrt{\frac{\delta a c_{n,\tau,\varepsilon}}
{\omega(\bc_{\tau,\varepsilon})}}\,
\boldsymbol A_{n,\tau,\varepsilon}
-
c_{n,\tau,\varepsilon}
\sqrt{\frac{ab}{\omega(\bc_{\tau,\varepsilon})}}\,
\boldsymbol B_{\tau,\varepsilon},
\]
and
\[
\boldsymbol J_{p,\tau,\varepsilon}
=
-
\sqrt{\frac{\delta b c_{p,\tau,\varepsilon}}
{\omega(\bc_{\tau,\varepsilon})}}\,
\boldsymbol A_{p,\tau,\varepsilon}
-
c_{p,\tau,\varepsilon}
\sqrt{\frac{ab}{\omega(\bc_{\tau,\varepsilon})}}\,
\boldsymbol B_{\tau,\varepsilon}.
\]
With these definitions, the approximate continuity equations take the form
\[
    \partial_t^\tau c_{i,\tau,\varepsilon}
    +
    \nabla\cdot\boldsymbol J_{i,\tau,\varepsilon}
    +
    \varepsilon\,\mathcal R_{i,\tau,\varepsilon}
    =
    0
    \quad\text{in }(H^\ell(\Omega))',
\]
where
\(
    \langle \mathcal R_{i,\tau,\varepsilon},\eta_i\rangle
    :=
    b(\mu_{i,\tau,\varepsilon},\eta_i),
    \;
    \eta_i\in H^\ell(\Omega).
\)

\begin{lem}[Uniform bounds]
\label{lem:global-uniform-bounds}
Assume that the approximate initial data are smooth and strictly positive, and
that
\[
    E(\bc^0)\le C .
\]
Let
\(
    (\bc_{\tau,\varepsilon},\bmu_{\tau,\varepsilon},
    \phi_{\tau,\varepsilon})
\)
be the approximate solutions constructed in
Subsection~\ref{subsec:entropy-variable-approximation}. Then there exists a
constant \(C>0\), independent of \(\tau\) and \(\varepsilon\), such that

\[
    \bc_{\tau,\varepsilon}
    \quad\text{is bounded in }
    L^\infty(0,T;L^2(\Omega)^2),
\]
\[
    \phi_{\tau,\varepsilon}
    \quad\text{is bounded in }
    L^\infty(0,T;H^1(\Omega)),
\]
\[
    \boldsymbol A_{n,\tau,\varepsilon},
    \boldsymbol A_{p,\tau,\varepsilon},
    \boldsymbol B_{\tau,\varepsilon}
    \quad\text{are bounded in }
    L^2(\Omega_T;\mathbb R^d),
\]
\[
    \int_0^T\int_\Omega
    \frac{
    |\nabla c_{n,\tau,\varepsilon}|^2
    +
    |\nabla c_{p,\tau,\varepsilon}|^2
    }
    {1+c_{n,\tau,\varepsilon}+c_{p,\tau,\varepsilon}}
    \,dxdt
    \le C,
\]
\[
    \bc_{\tau,\varepsilon}
    \quad\text{is bounded in }
    L^{4/3}(0,T;W^{1,4/3}(\Omega)^2),
\]
\[\chi_{n,\tau,\varepsilon}, \chi_{p,\tau,\varepsilon}  \quad\text{are bounded in }
L^{4/3}(0,T;W^{1,4/3}(\Omega)),\]
and
\[
    \boldsymbol J_{n,\tau,\varepsilon},
    \boldsymbol J_{p,\tau,\varepsilon}
    \quad\text{are bounded in }
    L^{4/3}(\Omega_T;\mathbb R^d).
\]
Moreover,
\[
    \partial_t^\tau c_{n,\tau,\varepsilon},
    \partial_t^\tau c_{p,\tau,\varepsilon}
    \quad\text{are bounded in }
    L^{4/3}(0,T;H^\ell(\Omega)'),
\]
and the regularization terms satisfy
\[
    \varepsilon\mathcal R_{n,\tau,\varepsilon},
    \varepsilon\mathcal R_{p,\tau,\varepsilon}
    \to0
    \quad\text{strongly in }
    L^2(0,T;H^\ell(\Omega)')
\]
as \(\varepsilon\to0\).
\end{lem}

\begin{proof}
Summing the discrete entropy inequality from
Lemma~\ref{lem:global-discrete-entropy} over \(k=1,\ldots,N\), we obtain
\[
    \sup_{1\le k\le N}E(\bc^k)
    +
    \tau\sum_{k=1}^N \mathcal D(\bc^k)
    +
    \varepsilon\tau
    \sum_{k=1}^N b(\bmu^k,\bmu^k)
    \le
    E(\bc^0).
\]
Equivalently,
\begin{equation}
    \sup_{0<t\le T}E(\bc_{\tau,\varepsilon}(t))
    +
    \int_0^T
    \mathcal D(\bc_{\tau,\varepsilon}(t))\,dt
    +
    \varepsilon
    \int_0^T
    b(\bmu_{\tau,\varepsilon},\bmu_{\tau,\varepsilon})\,dt
    \le C .
\label{eq:uniform-discrete-entropy-bound}
\end{equation}
In the pure Neumann constrained formulation, \(b\) is applied to the
zero-mean part of \(\bmu_{\tau,\varepsilon}\); we keep the same notation.

By Lemma~\ref{lem:energy-coercivity},
\[
    \|c_{n,\tau,\varepsilon}\|_{L^\infty(0,T;L^2(\Omega))}
    +
    \|c_{p,\tau,\varepsilon}\|_{L^\infty(0,T;L^2(\Omega))}
    \le C .
\]
The estimate for \(\phi_{\tau,\varepsilon}\) follows from the Poisson equation,
assumption \textup{(A3)}, and the \(L^\infty(0,T;L^2(\Omega))\) bound for
\(c_{p,\tau,\varepsilon}-c_{n,\tau,\varepsilon}\).

Next, by the definitions of
\(\boldsymbol A_{n,\tau,\varepsilon}\),
\(\boldsymbol A_{p,\tau,\varepsilon}\), and
\(\boldsymbol B_{\tau,\varepsilon}\), and by the dissipation decomposition
\eqref{eq:intro-dissipation-decomposition},
\[
    \mathcal D(\bc_{\tau,\varepsilon})
    =
    \int_\Omega
    \bigl(
    |\boldsymbol A_{n,\tau,\varepsilon}|^2
    +
    |\boldsymbol A_{p,\tau,\varepsilon}|^2
    +
    |\boldsymbol B_{\tau,\varepsilon}|^2
    \bigr)\,dx .
\]
Together with \eqref{eq:uniform-discrete-entropy-bound}, this gives
\[
    \boldsymbol A_{n,\tau,\varepsilon},
    \boldsymbol A_{p,\tau,\varepsilon},
    \boldsymbol B_{\tau,\varepsilon}
    \quad\text{bounded in }L^2(\Omega_T;\mathbb R^d).
\]

We now obtain the weighted gradient estimate. Define
\(
    \mathcal I_{\rm w}(\bc)
    :=
    \int_\Omega
    \frac{
    |\nabla c_n|^2+|\nabla c_p|^2
    }
    {1+c_n+c_p}
    \,dx .
\)
By Lemma~\ref{lem:global-entropy-production},
\[
    \mathcal D(\bc^k)
    \ge
    \kappa_{\rm w} \mathcal I_{\rm w}(\bc^k)
    -
    C\bigl(1+E(\bc^k)\bigr)^{3/2}.
\]
Since \(E(\bc^k)\) is uniformly bounded from
\eqref{eq:uniform-discrete-entropy-bound}, we have
\[
   \tau\sum_{k=1}^N \mathcal I_{\rm w}(\bc^k)
    \le
    \tau\sum_{k=1}^N \mathcal D(\bc^k)
    +
    C\tau\sum_{k=1}^N 1
    \le C .
\]
Hence
\[
    \int_0^T\int_\Omega
    \frac{
    |\nabla c_{n,\tau,\varepsilon}|^2
    +
    |\nabla c_{p,\tau,\varepsilon}|^2
    }
    {1+c_{n,\tau,\varepsilon}+c_{p,\tau,\varepsilon}}
    \,dxdt
    \le C .
\]
Combining this estimate with the
\(L^\infty(0,T;L^2(\Omega))\)-bound and
Lemma~\ref{lem:global-weighted-gradient}, we get
\[
    \bc_{\tau,\varepsilon}
    \quad\text{bounded in }
    L^{4/3}(0,T;W^{1,4/3}(\Omega)^2).
\]

We next obtain the corresponding bounds for the square-root weights. By
\eqref{eq:approximate-square-root-identities},
\begin{align*}
    2\nabla\chi_{n,\tau,\varepsilon}
    &=
    \boldsymbol A_{n,\tau,\varepsilon}
    -
    \chi_{n,\tau,\varepsilon}
    \nabla\Bigl(
        \log\omega(\bc_{\tau,\varepsilon})
        -\phi_{\tau,\varepsilon}
        +f_{nn}c_{n,\tau,\varepsilon}
        +f_{np}c_{p,\tau,\varepsilon}
    \Bigr),                                                \\
    2\nabla\chi_{p,\tau,\varepsilon}
    &=
    \boldsymbol A_{p,\tau,\varepsilon}
    -
    \chi_{p,\tau,\varepsilon}
    \nabla\Bigl(
        \log\omega(\bc_{\tau,\varepsilon})
        +\phi_{\tau,\varepsilon}
        +f_{np}c_{n,\tau,\varepsilon}
        +f_{pp}c_{p,\tau,\varepsilon}
    \Bigr).
\end{align*}
Since
\(
    0\le\chi_{n,\tau,\varepsilon}\le\sqrt{\delta a/b}
\)
and
\(
    0\le\chi_{p,\tau,\varepsilon}\le\sqrt{\delta b/a},
\)
while
\(
    \bigl|\nabla\log\omega(\bc_{\tau,\varepsilon})\bigr|
    \le
    \frac{
        b|\nabla c_{n,\tau,\varepsilon}|
        +
        a|\nabla c_{p,\tau,\varepsilon}|
    }{\delta},
\)
the bounds established above imply
\begin{align*}
    &\|\nabla\chi_{n,\tau,\varepsilon}\|_{L^{4/3}(\Omega_T)}
    +
    \|\nabla\chi_{p,\tau,\varepsilon}\|_{L^{4/3}(\Omega_T)}
    \\
    &\;\,\le
    C\Bigl(
        \|\boldsymbol A_{n,\tau,\varepsilon}\|_{L^2(\Omega_T)}
        +
        \|\boldsymbol A_{p,\tau,\varepsilon}\|_{L^2(\Omega_T)}
        +
        \|\nabla c_{n,\tau,\varepsilon}\|_{L^{4/3}(\Omega_T)}
        +
        \|\nabla c_{p,\tau,\varepsilon}\|_{L^{4/3}(\Omega_T)}
        +
        \|\nabla\phi_{\tau,\varepsilon}\|_{L^2(\Omega_T)}
    \Bigr).
\end{align*}
Here we used the embedding
\(L^2(\Omega_T)\hookrightarrow L^{4/3}(\Omega_T)\).
Together with the uniform \(L^\infty(\Omega_T)\)-bounds for
\(\chi_{n,\tau,\varepsilon}\) and \(\chi_{p,\tau,\varepsilon}\), this proves
that both weights are bounded in
\(L^{4/3}(0,T;W^{1,4/3}(\Omega))\).

The flux estimate follows from Lemma~\ref{lem:global-flux-estimate}, applied
to the approximate flux representation above. Thus
\[
    \boldsymbol J_{n,\tau,\varepsilon},
    \boldsymbol J_{p,\tau,\varepsilon}
    \quad\text{are bounded in }
    L^{4/3}(\Omega_T;\mathbb R^d).
\]

It remains to estimate the discrete time derivative. For
\(t\in((k-1)\tau,k\tau]\), the discrete equation gives, for every
\(\eta_i\in H^\ell(\Omega)\),
\[
    \int_\Omega
    \partial_t^\tau c_{i,\tau,\varepsilon}\eta_i\,dx
    -
    \int_\Omega
    \boldsymbol J_{i,\tau,\varepsilon}\cdot\nabla\eta_i\,dx
    +
    \varepsilon b(\mu_{i,\tau,\varepsilon},\eta_i)
    =
    0 .
\]
Since \(H^\ell(\Omega)\hookrightarrow W^{1,4}(\Omega)\),
\[
\left|
    \int_\Omega
    \boldsymbol J_{i,\tau,\varepsilon}\cdot\nabla\eta_i\,dx
\right|
\le
C
\|\boldsymbol J_{i,\tau,\varepsilon}\|_{L^{4/3}(\Omega)}
\|\eta_i\|_{H^\ell(\Omega)} .
\]
Thus, the flux term is bounded in
\(L^{4/3}(0,T;(H^\ell(\Omega))')\).

For the regularization term, Cauchy's inequality with respect to \(b\) gives
\[
    |\varepsilon b(\mu_{i,\tau,\varepsilon},\eta_i)|
    \le
    \varepsilon^{1/2}
    \bigl(
    \varepsilon b(\mu_{i,\tau,\varepsilon},
    \mu_{i,\tau,\varepsilon})
    \bigr)^{1/2}
    \|\eta_i\|_{H^\ell(\Omega)} .
\]
Therefore, by \eqref{eq:uniform-discrete-entropy-bound},
\[
    \|\varepsilon\mathcal R_{i,\tau,\varepsilon}\|_
    {L^2(0,T;H^\ell(\Omega)')}
    \le
    C\varepsilon^{1/2}.
\]
In particular,
\[
    \varepsilon\mathcal R_{i,\tau,\varepsilon}
    \to0
    \quad\text{strongly in }
    L^2(0,T;H^\ell(\Omega)')
\]
as \(\varepsilon\to0\), and the same term is uniformly bounded in
\(L^{4/3}(0,T;H^\ell(\Omega)')\). Combining the flux and regularization
estimates gives
\[
    \partial_t^\tau c_{i,\tau,\varepsilon}
    \quad\text{bounded in }
    L^{4/3}(0,T;H^\ell(\Omega)'),
    \qquad i=n,p .
\]
This completes the proof.
\end{proof}

We next extract compactness. We use the following discrete Aubin--Lions--Simon
compactness criterion for piecewise constant functions in time.

\begin{lem}[Discrete Aubin--Lions compactness]
\label{lem:discrete-aubin-lions}
Let
\(    X\hookrightarrow \mathcal B\hookrightarrow Y,
\)
where \(X\hookrightarrow\mathcal B\) is compact and
\(\mathcal B\hookrightarrow Y\) is continuous. Let \(v_\tau\) be piecewise
constant in time and set
\(
    \partial_t^\tau v_\tau
    =
    \frac{v_\tau-\sigma_\tau v_\tau}{\tau}.
\)
If $v_\tau$ is bounded in $L^p(0,T;X)$,
and $\partial_t^\tau v_\tau$
is bounded in $L^q(0,T;Y)$,
with \(1\le p,q\le\infty\) and \(p<\infty\), then \(v_\tau\) is relatively
compact in \(L^p(0,T;\mathcal B)\).
\end{lem}

\begin{proof}
This is the standard discrete version of the Aubin--Lions--Simon compactness
lemma; see, for instance,
\cite{simon1987compact,dreher2012compact}.
\end{proof}

\begin{cor}[Compactness of the approximate concentrations]
\label{cor:compactness-concentrations}
Let \(\tau\to0\) and let \(\varepsilon=\varepsilon(\tau)\to0\). Up to a
subsequence, there exist nonnegative functions \(c_n,c_p\) such that, for
\(i=n,p\),
\[
    c_{i,\tau,\varepsilon(\tau)}
    \to c_i
    \quad\text{strongly in }
    L^{4/3}(0,T;L^2(\Omega)),
\]
and almost everywhere in \(\Omega_T\). Moreover,
\[
    \sigma_\tau c_{i,\tau,\varepsilon(\tau)}
    \to c_i
    \quad\text{strongly in }
    L^{4/3}(0,T;L^2(\Omega)),
\]
\[
    c_{i,\tau,\varepsilon(\tau)}
    \rightharpoonup c_i
    \quad\text{weakly in }
    L^{4/3}(0,T;W^{1,4/3}(\Omega)),
\]
and
\[
    c_{i,\tau,\varepsilon(\tau)}
    \overset{*}{\rightharpoonup} c_i
    \quad\text{weakly-* in }
    L^\infty(0,T;L^2(\Omega)).
\]
Furthermore, with
\(
    \phi_{\tau,\varepsilon(\tau)}
    =
    \mathcal P
    \bigl(
    c_{p,\tau,\varepsilon(\tau)}
    -
    c_{n,\tau,\varepsilon(\tau)}
    \bigr),
\)
we have, up to a further subsequence,
\[
    \phi_{\tau,\varepsilon(\tau)}
    \rightharpoonup \phi
    \quad\text{weakly-* in }L^\infty(0,T;H^1(\Omega)),
\]
where
\(
    \phi=\mathcal P(c_p-c_n).
\)
\end{cor}

\begin{proof}
Apply Lemma~\ref{lem:discrete-aubin-lions} with
\[
    X=W^{1,4/3}(\Omega),
    \qquad
    \mathcal B=L^2(\Omega),
    \qquad
    Y=H^\ell(\Omega)'.
\]
Since \(d\le3\),
\(    W^{1,4/3}(\Omega)\hookrightarrow L^2(\Omega)
\)
is compact, while
\(
    L^2(\Omega)\hookrightarrow H^\ell(\Omega)'
\)
is continuous since \(H^\ell(\Omega)\hookrightarrow L^2(\Omega)\). The
uniform bounds from Lemma~\ref{lem:global-uniform-bounds} therefore imply
relative compactness of \(c_{i,\tau,\varepsilon(\tau)}\) in
\(
    L^{4/3}(0,T;L^2(\Omega)).
\)
This gives the strong convergence after extracting a subsequence.

Therefore, the weak convergence in
\(L^{4/3}(0,T;W^{1,4/3}(\Omega))\) and the weak-* convergence in
\(L^\infty(0,T; L^2(\Omega))\) follow from the corresponding uniform bounds.
After extracting a further subsequence, the strong convergence in
\(L^{4/3}(0,T;L^2(\Omega))\) implies almost everywhere convergence in
\(\Omega_T\). Since the approximate concentrations are nonnegative, the
limits \(c_n,c_p\) are nonnegative.

It remains to identify the time-shifted limit. The estimate
\[
    c_{i,\tau,\varepsilon(\tau)}
    -
    \sigma_\tau c_{i,\tau,\varepsilon(\tau)}
    =
    \tau\,\partial_t^\tau c_{i,\tau,\varepsilon(\tau)}
    \to0
    \quad\text{in }
    L^{4/3}(0,T;H^\ell(\Omega)')
\]
does not by itself imply strong convergence in \(L^2(\Omega)\). We therefore
apply the same compactness argument to the shifted sequence. The family
\(    \sigma_\tau c_{i,\tau,\varepsilon(\tau)}
\)
satisfies the same bounds as \(c_{i,\tau,\varepsilon(\tau)}\) in
\(L^\infty(0,T;L^2(\Omega)) \cap \)   $
    L^{4/3}(0,T;W^{1,4/3}(\Omega)).$
Indeed, it is obtained by shifting the discrete values by one time step, with
only the harmless initial interval involving the initial data. Its discrete
time derivative is, up to the same harmless endpoint convention, a time shift
of \(\partial_t^\tau c_{i,\tau,\varepsilon(\tau)}\); hence it is bounded in
\(
    L^{4/3}(0,T;H^\ell(\Omega)').
\)
Applying Lemma~\ref{lem:discrete-aubin-lions} again, we find that
\(    \sigma_\tau c_{i,\tau,\varepsilon(\tau)}
\)
is relatively compact in
\(
    L^{4/3}(0,T;L^2(\Omega)).
\)

Let \(\widetilde c_i\) be the strong
\(L^{4/3}(0,T;L^2(\Omega))\)-limit of an arbitrary convergent subsequence of
\(\sigma_\tau c_{i,\tau,\varepsilon(\tau)}\). Since
\(
    c_{i,\tau,\varepsilon(\tau)}
    -
    \sigma_\tau c_{i,\tau,\varepsilon(\tau)}
    =
    \tau\,\partial_t^\tau c_{i,\tau,\varepsilon(\tau)}
\)
and
\[
    \partial_t^\tau c_{i,\tau,\varepsilon(\tau)}
    \quad\text{is bounded in }
    L^{4/3}(0,T;H^\ell(\Omega)'),
\]
we have
\[
    c_{i,\tau,\varepsilon(\tau)}
    -
    \sigma_\tau c_{i,\tau,\varepsilon(\tau)}
    \to0
    \quad\text{strongly in }
    L^{4/3}(0,T;H^\ell(\Omega)').
\]
On the other hand,
\[
    c_{i,\tau,\varepsilon(\tau)}
    \to c_i
    \quad\text{strongly in }L^{4/3}(0,T;L^2(\Omega)),
\]
and the embedding
\(
    L^2(\Omega)\hookrightarrow H^\ell(\Omega)'
\)
is continuous. Therefore, the same subsequence satisfies
\[
    \sigma_\tau c_{i,\tau,\varepsilon(\tau)}
    \to c_i
    \quad\text{strongly in }
    L^{4/3}(0,T;H^\ell(\Omega)').
\]
Thus, \(\widetilde c_i=c_i\). Since every strongly convergent subsequence of
\(\sigma_\tau c_{i,\tau,\varepsilon(\tau)}\) has the same limit, the whole
shifted family converges strongly:
\[
    \sigma_\tau c_{i,\tau,\varepsilon(\tau)}
    \to c_i
    \quad\text{strongly in }
    L^{4/3}(0,T;L^2(\Omega)).
\]

Finally, the weak-* compactness of \(\phi_{\tau,\varepsilon(\tau)}\) in
\(L^\infty(0,T;H^1(\Omega))\) follows from
Lemma~\ref{lem:global-uniform-bounds}. Since the Poisson operator
\(\mathcal P\) is linear and continuous under assumption \textup{(A3)}, the
strong convergence of
\(
    c_{p,\tau,\varepsilon(\tau)}
    -
    c_{n,\tau,\varepsilon(\tau)}
\)
in \(L^{4/3}(0,T;L^2(\Omega))\) identifies the limit as
\(
    \phi=\mathcal P(c_p-c_n).
\)
\end{proof}


\subsection{Passage to the limit}
\label{subsec:passage-to-limit}

We now pass to the limit in the approximate solutions. Let
\(\tau\to0\) and choose \(\varepsilon=\varepsilon(\tau)\to0\). To simplify
notation, we write
\(
    \bc_\tau:=\bc_{\tau,\varepsilon(\tau)},
    \;
    \bmu_\tau:=\bmu_{\tau,\varepsilon(\tau)},
    \;
    \phi_\tau:=\phi_{\tau,\varepsilon(\tau)} ,
\)
and similarly for
\(
    \boldsymbol A_{n,\tau},\;
    \boldsymbol A_{p,\tau},\;
    \boldsymbol B_\tau,\;
    \boldsymbol J_{n,\tau},\;
    \boldsymbol J_{p,\tau}.
\)
By Corollary~\ref{cor:compactness-concentrations}, up to a subsequence,
\[
    c_{i,\tau}\to c_i
    \quad\text{strongly in }
    L^{4/3}(0,T;L^2(\Omega)),
    \qquad i=n,p,
\]
and almost everywhere in \(\Omega_T\). Moreover,
\[
    c_{i,\tau}\rightharpoonup c_i
    \quad\text{weakly in }
    L^{4/3}(0,T;W^{1,4/3}(\Omega)),
\]
and
\[
    c_{i,\tau}\overset{*}{\rightharpoonup} c_i
    \quad\text{weakly-* in }
    L^\infty(0,T;L^2(\Omega)).
\]
The limits are nonnegative. The electrostatic potentials satisfy
\[
    \phi_\tau=\mathcal P(c_{p,\tau}-c_{n,\tau})
    \rightharpoonup
    \phi=\mathcal P(c_p-c_n)
    \quad\text{weakly-* in }
    L^\infty(0,T;H^1(\Omega)).
\]

The uniform dissipation bound gives, up to a further subsequence,
\[
    \boldsymbol A_{n,\tau}\rightharpoonup \boldsymbol A_n,
    \qquad
    \boldsymbol A_{p,\tau}\rightharpoonup \boldsymbol A_p,
    \qquad
    \boldsymbol B_\tau\rightharpoonup \boldsymbol B \qquad \text{ weakly in } L^2(\Omega_T;\mathbb R^d).
\]
The flux bound gives
\[
    \boldsymbol J_{n,\tau}\rightharpoonup \boldsymbol J_n,
    \qquad
    \boldsymbol J_{p,\tau}\rightharpoonup \boldsymbol J_p  \qquad \text{ weakly in } L^{4/3}(\Omega_T;\mathbb R^d).
\]

We next establish the strong convergence of the nonlinear coefficients appearing
in the physical fluxes. Since
\(
    c_{i,\tau}\to c_i
    \;\text{ strongly in }
    L^{4/3}(0,T;L^2(\Omega))
\)
and \(\{c_{i,\tau}\}_\tau\) is bounded in
\(L^\infty(0,T;L^2(\Omega))\), interpolation gives
\[
    c_{i,\tau}\to c_i
    \quad\text{strongly in }L^2(\Omega_T),
    \qquad i=n,p .
\]
Together with the almost everywhere convergence and the bounds
\[
    0\le
    \left(
    \sqrt{\frac{\delta a c_{n,\tau}}{\omega(\bc_\tau)}}
    \right)^4
    \le C c_{n,\tau}^2,
    \qquad
    0\le
    \left(
    \sqrt{\frac{\delta b c_{p,\tau}}{\omega(\bc_\tau)}}
    \right)^4
    \le C c_{p,\tau}^2,
\]
and
\[
    0\le
    \left(
    c_{n,\tau}\sqrt{\frac{ab}{\omega(\bc_\tau)}}
    \right)^4
    \le C c_{n,\tau}^2,
    \qquad
    0\le
    \left(
    c_{p,\tau}\sqrt{\frac{ab}{\omega(\bc_\tau)}}
    \right)^4
    \le C c_{p,\tau}^2,
\]
we obtain
\[
    \sqrt{\frac{\delta a c_{n,\tau}}{\omega(\bc_\tau)}}
    \to
    \sqrt{\frac{\delta a c_n}{\omega(\bc)}},
    \qquad
    \sqrt{\frac{\delta b c_{p,\tau}}{\omega(\bc_\tau)}}
    \to
    \sqrt{\frac{\delta b c_p}{\omega(\bc)}}
\]
strongly in \(L^4(\Omega_T)\), and
\[
    c_{n,\tau}\sqrt{\frac{ab}{\omega(\bc_\tau)}}
    \to
    c_n\sqrt{\frac{ab}{\omega(\bc)}},
    \qquad
    c_{p,\tau}\sqrt{\frac{ab}{\omega(\bc_\tau)}}
    \to
    c_p\sqrt{\frac{ab}{\omega(\bc)}}
\]
strongly in \(L^4(\Omega_T)\) by means of the generalized Lebesgue dominated convergence theorem.

We next identify the weak limits of the weighted entropy gradients. By
Lemma~\ref{lem:global-uniform-bounds},
\(\chi_{n,\tau}\) and \(\chi_{p,\tau}\) are bounded in
\(L^{4/3}(0,T;W^{1,4/3}(\Omega))\). Together with the strong convergence above,
this yields
\[
    \chi_{i,\tau}
    \rightharpoonup
    \chi_i
    \quad\text{weakly in }
    L^{4/3}(0,T;W^{1,4/3}(\Omega)),
    \qquad i=n,p,
\]
where
\(
    \chi_n=\sqrt{\delta a c_n/\omega(\bc)}
\)
and
\(
    \chi_p=\sqrt{\delta b c_p/\omega(\bc)}.
\)
In particular,
\(
    \nabla\chi_{i,\tau}
    \rightharpoonup
    \nabla\chi_i
\)
weakly in \(L^{4/3}(\Omega_T;\mathbb R^d)\).

Moreover, since
\(c_{p,\tau}-c_{n,\tau}\to c_p-c_n\) strongly in \(L^2(\Omega_T)\),
the continuity of the Poisson operator in \textup{(A3)} gives
\(
    \phi_\tau\to\phi
\)
strongly in \(L^2(0,T;H^2(\Omega))\). Hence
\(
    \chi_{i,\tau}\nabla\phi_\tau
    \to
    \chi_i\nabla\phi
\)
strongly in \(L^{4/3}(\Omega_T;\mathbb R^d)\).

Since \(\omega(\bc_\tau)\ge\delta\), both \(\chi_{i,\tau}\) and \(\chi_{i,\tau}/\omega(\bc_\tau)\) are uniformly bounded and converge almost everywhere to \(\chi_i\) and \(\chi_i/\omega(\bc)\), respectively. The weak convergence of \(\nabla c_{n,\tau}\) and \(\nabla c_{p,\tau}\) in \(L^{4/3}(\Omega_T;\mathbb R^d)\) therefore implies
\[
\begin{aligned}
    &\chi_{n,\tau}
    \nabla\Bigl(
        \log\omega(\bc_\tau)
        +f_{nn}c_{n,\tau}
        +f_{np}c_{p,\tau}
    \Bigr)
    \rightharpoonup
    \chi_n
    \nabla\Bigl(
        \log\omega(\bc)
        +f_{nn}c_n
        +f_{np}c_p
    \Bigr),
    \\
    &\chi_{p,\tau}
    \nabla\Bigl(
        \log\omega(\bc_\tau)
        +f_{np}c_{n,\tau}
        +f_{pp}c_{p,\tau}
    \Bigr)
    \rightharpoonup
    \chi_p
    \nabla\Bigl(
        \log\omega(\bc)
        +f_{np}c_n
        +f_{pp}c_p
    \Bigr)
\end{aligned}
\]
weakly in \(L^{4/3}(\Omega_T;\mathbb R^d)\). Here the passage to the
limit follows by testing against functions in \(L^4(\Omega_T;\mathbb R^d)\)
and using dominated convergence for the bounded coefficients.

Passing to the limit in
\eqref{eq:approximate-square-root-identities}, we conclude that
\[
\begin{aligned}
    \boldsymbol A_n
    &=
    2\nabla\chi_n
    +
    \chi_n
    \nabla\Bigl(
        \log\omega(\bc)
        -\phi
        +f_{nn}c_n
        +f_{np}c_p
    \Bigr),
    \\
    \boldsymbol A_p
    &=
    2\nabla\chi_p
    +
    \chi_p
    \nabla\Bigl(
        \log\omega(\bc)
        +\phi
        +f_{np}c_n
        +f_{pp}c_p
    \Bigr)
\end{aligned}
\]
in \(L^{4/3}(\Omega_T;\mathbb R^d)\). Since
\(\boldsymbol A_n,\boldsymbol A_p\in L^2(\Omega_T;\mathbb R^d)\), these are
precisely the constitutive identities required in
Definition~\ref{def:finite-energy-weak-solution}.

The same argument also gives
\(
    \sqrt{\frac{b c_{n,\tau}}{\delta}}
    \to
    \sqrt{\frac{b c_n}{\delta}},
    \;
    \sqrt{\frac{a c_{p,\tau}}{\delta}}
    \to
    \sqrt{\frac{a c_p}{\delta}}
\)
strongly in \(L^4(\Omega_T)\). Passing to the limit in the approximate
compatibility identity \eqref{eq:approximate-compatibility}, using the
strong--weak convergence in \(L^4\times L^2\), yields
\begin{equation}
    \boldsymbol B
    =
    \sqrt{\frac{b c_n}{\delta}}\,
    \boldsymbol A_n
    +
    \sqrt{\frac{a c_p}{\delta}}\,
    \boldsymbol A_p
    \quad\text{in }L^{4/3}(\Omega_T;\mathbb R^d).
\label{eq:limit-compatibility}
\end{equation}

We now identify the weak limits of the fluxes. From the approximate flux
representations,
\[
\boldsymbol J_{n,\tau}
=
-
\sqrt{\frac{\delta a c_{n,\tau}}{\omega(\bc_\tau)}}\,
\boldsymbol A_{n,\tau}
-
c_{n,\tau}
\sqrt{\frac{ab}{\omega(\bc_\tau)}}\,
\boldsymbol B_\tau,
\quad
\boldsymbol J_{p,\tau}
=
-
\sqrt{\frac{\delta b c_{p,\tau}}{\omega(\bc_\tau)}}\,
\boldsymbol A_{p,\tau}
-
c_{p,\tau}
\sqrt{\frac{ab}{\omega(\bc_\tau)}}\,
\boldsymbol B_\tau .
\]
The strong convergence of the coefficients in \(L^4(\Omega_T)\), together
with the weak convergence of
\(\boldsymbol A_{n,\tau}\), \(\boldsymbol A_{p,\tau}\), and
\(\boldsymbol B_\tau\) in \(L^2(\Omega_T;\mathbb R^d)\), gives
\[
\boldsymbol J_n
=
-
\sqrt{\frac{\delta a c_n}{\omega(\bc)}}\,
\boldsymbol A_n
-
c_n
\sqrt{\frac{ab}{\omega(\bc)}}\,
\boldsymbol B,
\qquad
\boldsymbol J_p
=
-
\sqrt{\frac{\delta b c_p}{\omega(\bc)}}\,
\boldsymbol A_p
-
c_p
\sqrt{\frac{ab}{\omega(\bc)}}\,
\boldsymbol B.
\]
In particular,
\(
    \boldsymbol J_n,\boldsymbol J_p
    \in L^{4/3}(\Omega_T;\mathbb R^d).
\)

We next pass to the limit in the continuity equations. From Lemma~\ref{lem:global-uniform-bounds}, the discrete time derivatives are bounded in \(L^{4/3}(0,T;H^\ell(\Omega)'). \) Together with the strong convergence \[ c_{i,\tau}\to c_i \quad\text{in }L^{4/3}(0,T;L^2(\Omega)), \] this implies, after extracting a subsequence if necessary, that \[ \partial_t^\tau c_{i,\tau} \rightharpoonup \partial_t c_i \quad\text{weakly in } L^{4/3}(0,T;H^\ell(\Omega)'). \]
The approximate equations can be written as
\[
    \partial_t^\tau c_{i,\tau}
    +
    \nabla\cdot\boldsymbol J_{i,\tau}
    +
    \varepsilon(\tau)\mathcal R_{i,\tau}
    =
    0
    \quad\text{in }H^\ell(\Omega)'.
\]
By Lemma~\ref{lem:global-uniform-bounds},
\[
    \varepsilon(\tau)\mathcal R_{i,\tau}
    \to0
    \quad\text{strongly in }
    L^2(0,T;H^\ell(\Omega)').
\]
Therefore, passing to the limit gives
\[
    \partial_t c_i+\nabla\cdot\boldsymbol J_i=0
    \quad\text{in }\mathcal D'(0,T;H^\ell(\Omega)').
\]
Equivalently, for every
\(
    \varphi_i\in L^4(0,T;W^{1,4}(\Omega)),
\)
we have
\begin{equation}
    \int_0^T
    \langle \partial_t c_i,\varphi_i\rangle\,dt
    -
    \int_{\Omega_T}
    \boldsymbol J_i\cdot\nabla\varphi_i\,dxdt
    =
    0,
    \qquad i=n,p .
\label{eq:limit-continuity-equation}
\end{equation}
Indeed, the identity is first obtained for \(H^\ell\)-test functions and then
extended by density, since
\(\boldsymbol J_i\in L^{4/3}(\Omega_T;\mathbb R^d)\). Consequently,
\(
    \partial_t c_i
    =
    -\nabla\cdot\boldsymbol J_i
    \in L^{4/3}(0,T;W^{1,4}(\Omega)').
\)

It remains to pass to the limit in the energy inequality. Let \(t\in(0,T]\)
be such that
\[
    c_{i,\tau}(t)\to c_i(t)
    \quad\text{strongly in }L^2(\Omega),
    \qquad i=n,p,
\]
which holds for almost every \(t\) after extracting a subsequence. Choose
\(m=m(\tau,t)\) such that
\(
    t\in((m-1)\tau,m\tau],
\)
so that
\(
    \bc_\tau(t)=\bc^m.
\)
The discrete entropy inequality gives
\[
    E(\bc_\tau(t))
    +
    \int_0^{m\tau}
    \int_\Omega
    \bigl(
    |\boldsymbol A_{n,\tau}|^2
    +
    |\boldsymbol A_{p,\tau}|^2
    +
    |\boldsymbol B_\tau|^2
    \bigr)\,dxds
    \le
    E(\bc^0).
\]
Since \(m\tau\to t\), the weak lower semicontinuity of the \(L^2\)-norm gives
\[
\begin{aligned}
\int_0^t\int_\Omega
\bigl(
|\boldsymbol A_n|^2
+
|\boldsymbol A_p|^2
+
|\boldsymbol B|^2
\bigr)\,dxds
\le
\liminf_{\tau\to0}
\int_0^{m\tau}
\int_\Omega
\bigl(
|\boldsymbol A_{n,\tau}|^2
+
|\boldsymbol A_{p,\tau}|^2
+
|\boldsymbol B_\tau|^2
\bigr)\,dxds .
\end{aligned}
\]
For the energy term, the strong \(L^2(\Omega)\)-convergence of
\(c_{i,\tau}(t)\), the convexity of the local free energy density, and the
continuity of the Poisson operator give
\[
    E(\bc(t))
    \le
    \liminf_{\tau\to0}E(\bc_\tau(t)).
\]
Therefore, for almost every \(t\in(0,T]\),
\begin{equation}
    E(\bc(t))
    +
    \int_0^t\int_\Omega
    \bigl(
    |\boldsymbol A_n|^2
    +
    |\boldsymbol A_p|^2
    +
    |\boldsymbol B|^2
    \bigr)\,dxds
    \le
    E(\bc_0).
\label{eq:limit-energy-inequality}
\end{equation}

We finally identify the initial condition. From
\[
    c_i\in L^{4/3}(0,T;W^{1,4/3}(\Omega)),
    \qquad
    \partial_t c_i\in L^{4/3}(0,T;W^{1,4}(\Omega)'),
\]
the function \(c_i\) admits a representative in
\(
    C([0,T];W^{1,4}(\Omega)').
\)
Since the approximate initial data converge to \(c_{i,0}\) in \(L^2(\Omega)\)
and \(c_{i,\tau}\to c_i\) in the sense above, we obtain
\[
    c_i(0)=c_{i,0}
    \quad\text{in }W^{1,4}(\Omega)' ,
    \qquad i=n,p .
\]
By the preceding identification, the fields \(\boldsymbol A_n,\boldsymbol A_p\) satisfy \eqref{eq:weak-weighted-entropy-gradients}, while \(\boldsymbol B\) is given by \eqref{eq:weak-collective-field}. Together with the flux identities, the continuity equations, the initial condition, and the energy inequality, this shows that \((\bc,\phi)\) is a finite-energy weak solution in the sense of Definition~\ref{def:finite-energy-weak-solution} for smooth strictly positive initial data.

For general finite-energy initial data \(\bc_0\), choose smooth strictly
positive approximations
\(
    \bc_0^m=(c_{n,0}^m,c_{p,0}^m)^\top
\)
such that
\(
    c_{i,0}^m\to c_{i,0}
    \;\text{ in }L^2(\Omega),
    \;
    c_{i,0}^m\ge0,
\)
and
\(    E(\bc_0^m)\to E(\bc_0).
\)
In the pure Neumann case the approximations are chosen to preserve the species
masses, so that the Poisson compatibility condition holds. 
Applying the preceding construction to \(\bc_0^m\) gives finite-energy weak solutions with estimates depending only on \(E(\bc_0^m)\). Passing to the limit \(m\to\infty\) by the same compactness argument, and repeating the constitutive identification above, we obtain a finite-energy weak solution \((\bc,\phi)\) with initial data \(\bc_0\). The lower-semicontinuity argument also yields \eqref{eq:limit-energy-inequality} with \(E(\bc_0)\) on the right-hand side.


\subsection{Degenerate rank-one steric interactions}
\label{subsec:rank-one-steric}

We conclude this section by discussing the role of the positive definiteness
assumption on \(F\). In the proof of
Theorem~\ref{thm:global-weak-intro}, the positive definite steric energy gives
species-wise \(L^2\)-control, which is then combined with the entropy
production to obtain compactness of the individual concentrations.  This argument no longer applies in the rank-one case: the energy controls only the total
density, while the charge mode remains uncontrolled.

Assume in this subsection that
\(
    F
    =
    f
    \begin{pmatrix}
    1&1\\
    1&1
    \end{pmatrix},
    \; f>0.
\)
Set
\(
    u=c_n+c_p,
    \;
    \rho=c_p-c_n .
\)
Then the steric part of the free energy is
\(
    \frac f2 u^2.
\)
Thus, the rank-one steric energy sees \(u\), but not \(c_n\) and \(c_p\)
separately. Equivalently,
\(
    \partial_{c_n}\left(\frac f2u^2\right)
    =
    \partial_{c_p}\left(\frac f2u^2\right)
    =
    fu,
\)
and therefore
\(
    \nabla c_n\cdot\nabla(fu)
    +
    \nabla c_p\cdot\nabla(fu)
    =
    f|\nabla u|^2.
\)
The total density remains visible in the steric contribution, whereas the
charge mode \(\rho\) is invisible to the rank-one quadratic energy.

\begin{prop}[Compactness of the total density in the rank-one case]
\label{prop:rank-one-total-density}
Assume
\(    F
    =
    f
    \begin{pmatrix}
    1&1\\
    1&1
    \end{pmatrix},
    \; f>0.
\)
Let
\(
    \bc^m=(c_n^m,c_p^m)^\top
\)
be a sequence of smooth positive solutions satisfying the physical flux
formulation
\[
    \partial_t c_i^m+\nabla\cdot\boldsymbol J_i^m=0,
    \qquad
    \boldsymbol J^m=-M(\bc^m)\nabla\bmu^m,
    \qquad i=n,p,
\]
with the corresponding no-flux boundary conditions. Suppose that, for every
\(T>0\), there exists a positive constant $C_T$ independent of $m$ such that
\[
    \sup_{0<t<T}E(\bc^m(t))
    +
    \int_0^T \mathcal D(\bc^m(t))\,dt
    \le C_T,
\]
and that the species masses are fixed. Define
\(    u^m=c_n^m+c_p^m,
    \;
    \rho^m=c_p^m-c_n^m .
\)
Then
\[
    u^m
    \quad\text{is bounded in}\quad
    L^\infty(0,T;L^2(\Omega))
    \cap
    L^2(0,T;H^1(\Omega)).
\]
Moreover,
\[
    \partial_t u^m
    \quad\text{is bounded in}\quad
    L^{4/3}(0,T;W^{1,4}(\Omega)'),
\]
and hence, up to a subsequence,
\[
    u^m\to u
    \qquad
    \text{strongly in }L^2(0,T;L^2(\Omega)).
\]
\end{prop}

\begin{proof}
The \(L^\infty(0,T;L^2(\Omega))\)-bound for \(u^m\) follows directly from
the energy bound. Indeed, in the rank-one case, the steric energy contains
\(
    \frac f2\int_\Omega |u^m|^2\,dx .
\)
Since the ideal entropy is bounded from below and the electrostatic energy is
nonnegative, the uniform energy bound gives
\(    \|u^m\|_{L^\infty(0,T;L^2(\Omega))}
    \le C_T.
\)

We next estimate \(\nabla u^m\). In the rank-one case,
\[
    \mu_n^m
    =
    \ln c_n^m+1-\phi^m+fu^m,
    \qquad
    \mu_p^m
    =
    \ln c_p^m+1+\phi^m+fu^m.
\]
Therefore
\[
\begin{aligned}
c_n^m\nabla\mu_n^m+c_p^m\nabla\mu_p^m
&=
\nabla c_n^m+\nabla c_p^m
+
(c_p^m-c_n^m)\nabla\phi^m
+
f(c_n^m+c_p^m)\nabla u^m
\\
&=
(1+fu^m)\nabla u^m+\rho^m\nabla\phi^m .
\end{aligned}
\]
Using the collective mode in the entropy production
\eqref{eq:intro-dissipation-decomposition}, we get
\[
\mathcal D(\bc^m)
\ge
\int_\Omega
\frac{ab}{\omega(\bc^m)}
\left|
(1+fu^m)\nabla u^m+\rho^m\nabla\phi^m
\right|^2
\,dx .
\]
Since
\(
    \omega(\bc^m)
    =
    \delta+a c_p^m+b c_n^m,
\)
and \(a,b,\delta>0\), we have
\[
    c(1+u^m)
    \le
    \omega(\bc^m)
    \le
    C(1+u^m)
\]
for constants \(c,C>0\) independent of \(m\). Hence, Young's inequality gives
\[
\mathcal D(\bc^m)
\ge
c
\int_\Omega
\frac{(1+fu^m)^2}{1+u^m}
|\nabla u^m|^2\,dx
-
C
\int_\Omega
\frac{|\rho^m|^2}{1+u^m}
|\nabla\phi^m|^2\,dx .
\]
Noting that \(c_n^m,c_p^m\ge0\),
 we have \(
    |\rho^m|\le u^m,
\)
and therefore
\(
    \frac{|\rho^m|^2}{1+u^m}
    \le
    u^m.
\)
Thus
\[
\mathcal D(\bc^m)
\ge
c
\int_\Omega
\frac{(1+fu^m)^2}{1+u^m}
|\nabla u^m|^2\,dx
-
C
\int_\Omega
u^m|\nabla\phi^m|^2\,dx .
\]

The electrostatic term is controlled by the Poisson regularity assumption.
Since
\(
    -\Delta\phi^m=\rho^m
\)
and \(|\rho^m|\le u^m\), Lemma~\ref{lem:poisson-operator} gives
\[
    \|\nabla\phi^m\|_{L^6(\Omega)}
    \le
    C\bigl(1+\|\rho^m\|_{L^2(\Omega)}\bigr)
    \le
    C\bigl(1+\|u^m\|_{L^2(\Omega)}\bigr).
\]
Using \(d\le3\), H\"older's inequality, and the boundedness of \(\Omega\), we
obtain
\[
\begin{aligned}
\int_\Omega
u^m|\nabla\phi^m|^2\,dx
\le
\|u^m\|_{L^{3/2}(\Omega)}
\|\nabla\phi^m\|_{L^6(\Omega)}^2
\le
C
\bigl(1+\|u^m\|_{L^2(\Omega)}^3\bigr),
\end{aligned}
\]
where the right-hand side is uniformly bounded on the energy sublevels. Since
\[
    \frac{(1+fu)^2}{1+u}
    \ge c_0>0
    \qquad\text{for all }u\ge0,
\]
the dissipation bound implies
\[
    \int_0^T\int_\Omega
    |\nabla u^m|^2\,dxdt
    \le C_T.
\]
Consequently, $u^m$ is bounded in $L^2(0,T;H^1(\Omega))$.

It remains to estimate the time derivative. Since
\[
    \partial_t u^m
    =
    \partial_t c_n^m+\partial_t c_p^m
    =
    -\nabla\cdot(\boldsymbol J_n^m+\boldsymbol J_p^m),
\]
it suffices to control the total flux. The energy bound gives
\(u^m\in L^\infty(0,T;L^2(\Omega))\), and since
\(
    0\le c_n^m,c_p^m\le u^m,
\)
we also have
\[
    c_n^m,c_p^m
    \quad\text{bounded in}\quad
    L^\infty(0,T;L^2(\Omega)).
\]
Together with the \(L^2\)-control of the dissipation fields from
\eqref{eq:intro-dissipation-decomposition}, Lemma~\ref{lem:global-flux-estimate}
therefore yields
\[
    \boldsymbol J_n^m,\boldsymbol J_p^m
    \quad\text{bounded in}\quad
    L^{4/3}(\Omega_T;\mathbb R^d).
\]
Hence
\[
    \partial_t u^m
    \quad\text{is bounded in}\quad
    L^{4/3}(0,T;W^{1,4}(\Omega)').
\]

By the Aubin--Lions--Simon compactness lemma
\cite{simon1987compact}, applied to
\(
    H^1(\Omega)
    \hookrightarrow L^2(\Omega)
    \hookrightarrow W^{1,4}(\Omega)',
\)
with the first embedding compact, we obtain, up to a subsequence,
\[
    u^m\to u
    \qquad
    \text{strongly in }L^2(0,T;L^2(\Omega)).
\]
This proves the proposition.
\end{proof}

\begin{remark}[Lack of compactness in the charge mode]
\label{rem:rank-one-missing-charge}
Proposition~\ref{prop:rank-one-total-density} gives compactness only for the
total density
\(
    u=c_n+c_p.
\)
It does not provide strong compactness of the charge mode
\(
    \rho=c_p-c_n.
\)
Since,
\(
    c_n=\frac{u-\rho}{2},
    \;
    c_p=\frac{u+\rho}{2},
\)
thus strong compactness of \(u\) alone does not identify the individual
species. This is the obstruction to extending the proof of
Theorem~\ref{thm:global-weak-intro} directly to the rank-one case.

The obstruction appears precisely in the nonlinear fluxes. The nonlinear flux
coefficients contain the individual concentrations through terms such as
\(
    \sqrt{\frac{c_i}{\omega(\bc)}},
    \;
    \frac{c_i}{\sqrt{\omega(\bc)}},
    \; i=n,p.
\)
Strong compactness of \(u\) does not determine the strong limits of these
coefficients, because \(c_n\) and \(c_p\) may oscillate through the unresolved
charge mode \(\rho\). Thus, the nonlinear fluxes in the finite-energy weak
formulation cannot be identified from compactness of \(u\) alone.

The same degeneracy is visible by regularizing the steric matrix. Let
\(
    F_\sigma
    =
    \begin{pmatrix}
    f+\sigma & f\\
    f & f+\sigma
    \end{pmatrix},
\)
where $\sigma>0$.
The eigenvalues of \(F_\sigma\) are
\(
    \sigma,
    \;
    2f+\sigma.
\)
Hence, the coercivity constant satisfies
\[
    \alpha_{F_\sigma}=\sigma\to0
    \qquad\text{as }\sigma\downarrow0.
\]
Accordingly, the constants in the species-wise coercivity and compactness
estimates degenerate in the rank-one limit. In particular, although for each
\(\sigma>0\) one obtains a bound of the form
\[
    \int_0^T\int_\Omega
    \frac{|\nabla c_n|^2+|\nabla c_p|^2}
    {1+c_n+c_p}
    \,dxdt
    \le C_\sigma,
\]
where the constant \(C_\sigma\) is not uniformly bounded as \(\sigma\downarrow0\). A full
finite-energy weak theory in the rank-one case would therefore require an
additional mechanism controlling the missing charge mode.
\end{remark}

\section{Long-time behavior and sharp entropy production}
\label{sec:long-time}

We now return to the dissipation functional \(\mathcal D\) in the energy identity \eqref{eq:intro-energy-dissipation}. In Section~\ref{sec:global-weak}, the decomposition \eqref{eq:intro-dissipation-decomposition} was used to obtain compactness and construct global finite-energy weak solutions. Under the pure Neumann equal-mass assumption \textup{(A5)}, we use the same entropy-production structure to study relaxation. Throughout this section, \(\phi=\phi[\bc]\) denotes the normalized Neumann solution of \( -\Delta\phi=c_p-c_n, \; \partial_\nu\phi=0, \; \int_\Omega\phi\,dx=0 . \) The homogeneous equilibrium is \(\bc^\infty=(m,m)^\top, \; \phi^\infty=0, \) and we write \( \mathcal H(\bc|\bc^\infty) := E(\bc)-E(\bc^\infty) \) for the relative entropy. The section has three parts. We first prove a sublevel entropy-entropy production inequality for positive admissible states. We then apply this estimate to the mass-preserving entropy approximation and discuss the extension of the entropy-production inequality to general finite-energy states. Finally, we identify the sharp linearized entropy-production constant in the small-sublevel limit.

\subsection{Sublevel entropy-entropy production}
\label{subsec:sublevel-eep}

We begin with the static estimate underlying the long-time analysis. In this subsection, an admissible state means a pair \( \bc=(c_n,c_p)^\top\in W^{1,\infty}(\Omega)^2, \; c_n,c_p>0 \;\text{ a.e. in }\Omega, \) satisfying the mass constraints in \textup{(A5)}. For such states, \(\mathcal D(\bc)\) is understood by the smooth expression \eqref{eq:intro-dissipation-decomposition}. Using the mass constraints, the linear terms in the expansion of \(E(\bc)-E(\bc^\infty)\) cancel. Hence, 
\begin{equation}
\begin{aligned} 
\mathcal H(\bc|\bc^\infty) = \sum_{i=n,p} \int_\Omega \left( c_i\ln\frac{c_i}{m}-c_i+m \right)\,dx + \frac12\int_\Omega (\bc-\bc^\infty)^\top F(\bc-\bc^\infty)\,dx + \frac12\int_\Omega |\nabla\phi|^2\,dx . \end{aligned} \label{eq:relative-entropy} \end{equation} Since \(F\) is symmetric positive definite and \( s\ln\frac{s}{m}-s+m\ge0 \; (\text{for } s\ge0), \) there exists \(C>0\) such that \begin{equation} \|\bc-\bc^\infty\|_{L^2(\Omega)}^2 + \|\nabla\phi\|_{L^2(\Omega)}^2 \le C\mathcal H(\bc|\bc^\infty). \label{eq:relative-entropy-controls-L2} \end{equation}

\begin{lem}[Uniqueness of equilibrium]
\label{lem:equilibrium-unique}
In the mass class \textup{(A5)}, the unique minimizer of \(E\) is
\(
    \bc^\infty=(m,m)^\top,
    \;
    \phi^\infty=0 .
\)
\end{lem}

\begin{proof}
The ideal entropy is convex, the steric part is strictly convex because
\(F>0\), and the electrostatic energy is a nonnegative quadratic form of the
charge density. Hence \(E\) is strictly convex on the admissible mass class.
The state \(\bc^\infty=(m,m)^\top\) satisfies the mass constraints, gives
\(\phi^\infty=0\), and has constant electrochemical potentials. Therefore, it
is a critical point of \(E\) under the two mass constraints. Strict convexity
then gives uniqueness.
\end{proof}

We next prove the local estimate near \(\bc^\infty\). The proof uses only
smallness in relative entropy; no pointwise smallness is assumed.

\begin{lem}[Local entropy-entropy production inequality]
\label{lem:local-eep}
There exist constants \(\eta_0>0\) and \(\lambda_0>0\), depending only on
\(\Omega,m,a,b,\delta\), and \(F\), such that every admissible state satisfying
\(
    \mathcal H(\bc|\bc^\infty)\le \eta_0
\)
obeys
\[
    \mathcal D(\bc)
    \ge
    \lambda_0\mathcal H(\bc|\bc^\infty).
\]
\end{lem}

\begin{proof}
Suppose the assertion is false. Then there exists a sequence of admissible
states \(\bc^k=(c_n^k,c_p^k)^\top\) such that
\(
    \mathcal H_k:=\mathcal H(\bc^k|\bc^\infty)\to0,
    \;
    \frac{\mathcal D(\bc^k)}{\mathcal H_k}\to0.
\)

Set
\(
    \varepsilon_k:=\mathcal H_k^{1/2},
    \;
    \bq^k:=\frac{\bc^k-\bc^\infty}{\varepsilon_k},
    \;
    \psi^k:=\frac{\phi^k}{\varepsilon_k}.
\)
Then
\(
    \int_\Omega q_n^k\,dx
    =
    \int_\Omega q_p^k\,dx
    =
    0,
\)
and, with
\(
    \bz:=(-1,1)^\top,
\)
we have
\(
    -\Delta\psi^k=\bz\cdot\bq^k,
    \;
    \partial_\nu\psi^k=0,
    \;
    \int_\Omega\psi^k\,dx=0 .
\)

By \eqref{eq:relative-entropy-controls-L2},
\(
    \bq^k
    \;\text{ is bounded in }L^2(\Omega)^2,
    \;
    \psi^k
    \;\text{ is bounded in }H^2(\Omega).
\)

We claim that \(\bq^k\) is compact in \(L^2(\Omega)^2\). The weighted gradient
estimate behind Lemma~\ref{lem:global-weighted-gradient}, applied to a single
state, gives
\[
    \int_\Omega
    \frac{|\nabla c_n^k|^2+|\nabla c_p^k|^2}
         {1+c_n^k+c_p^k}
    \,dx
    \le
    C\mathcal D(\bc^k) + C\int_\Omega
(1+c_n^k+c_p^k)|\nabla\phi^k|^2\,dx.
\]
Moreover, since
\(
    c_p^k-c_n^k=\varepsilon_k(q_p^k-q_n^k)\) and as \(
    \|q_p^k-q_n^k\|_{L^2}\le C,
\)
the homogeneous Poisson estimate in \textup{(A3)} gives
\[
    \|\nabla\phi^k\|_{L^6}
    +
    \|\nabla\phi^k\|_{L^2}
    \le C\varepsilon_k .
\]
Together with the uniform \(L^2\)-bound on \(c_n^k+c_p^k\), this yields
\[
    \int_\Omega
    (1+c_n^k+c_p^k)|\nabla\phi^k|^2\,dx
    \le C\varepsilon_k^2 .
\]
Since
\(
    \mathcal D(\bc^k)=o(\varepsilon_k^2),
    \;
    \|\nabla\phi^k\|_{L^2}\le C\varepsilon_k,
\)
we obtain
\[
    \int_\Omega
    \frac{|\nabla q_n^k|^2+|\nabla q_p^k|^2}
         {1+c_n^k+c_p^k}
    \,dx
    \le C .
\]
Together with the \(L^2\)-bound on \(\bc^k\), this implies
\[
    \|\nabla q_i^k\|_{L^{4/3}(\Omega)}
    \le C,
    \qquad i=n,p,
\]
by H\"older's inequality. Hence \(\bq^k\) is bounded in
\(W^{1,4/3}(\Omega)^2\). Since \(d\le3\), the embedding
\(W^{1,4/3}(\Omega)\hookrightarrow L^2(\Omega)\) is compact. Thus, up to a
subsequence,
\(
    \bq^k\to\bq
    \; \text{ strongly in }L^2(\Omega)^2,
\) \(    \psi^k\to\psi
    \;\text{ strongly in }H^2(\Omega).
\)

We now identify the limiting linearized stationary system. Let
\(
    C^k:=\operatorname{diag}(c_n^k,c_p^k),
    \;
    \bmu^k:=\bmu(\bc^k,\phi^k).
\)
Since constants disappear under gradients, the vector
\(
    \boldsymbol g^k
    :=
    \begin{pmatrix}
    c_n^k\nabla\mu_n^k\\[1mm]
    c_p^k\nabla\mu_p^k
    \end{pmatrix}
\)
can be written as
\(
    \boldsymbol g^k
    =
    (I+C^kF)\nabla(\bc^k-\bc^\infty)
    +
    C^k\bz\nabla\phi^k .
\)
The dissipation estimate gives
\[
    \left\|
    \frac{\boldsymbol g^k}{\varepsilon_k}
    \right\|_{L^1(\Omega)}
    \le
    C\left(
    \frac{\mathcal D(\bc^k)}{\varepsilon_k^2}
    \right)^{1/2}
    \to0.
\]
Dividing the preceding identity for \(\boldsymbol g^k\) by \(\varepsilon_k\), we obtain
\(
    (I+C^kF)\nabla\bq^k
    +
    C^k\bz\nabla\psi^k
    \to0
    \;\text{ in }L^1(\Omega).
\)

Because \(F>0\), the matrices \(I+C^kF\) are invertible for
\(c_n^k,c_p^k\ge0\). Moreover, the coefficient maps
\[
    \bc\mapsto (I+C(\bc)F)^{-1},
    \qquad
    \bc\mapsto (I+C(\bc)F)^{-1}C(\bc),
    \qquad
    C(\bc):=\operatorname{diag}(c_n,c_p),
\]
are globally bounded and continuous on \([0,\infty)^2\). Indeed,
\[
    \det(I+C(\bc)F)
    =
    1+c_n f_{nn}+c_p f_{pp}+c_nc_p\det F>0,
\]
and the entries are ratios of polynomial expressions divided by this
determinant. Since \(\bc^k\to\bc^\infty\) strongly in \(L^2\) and a.e., we may
pass to the limit in the inverted relation and obtain
\[
    (I+mF)\nabla\bq
    +
    m\bz\nabla\psi
    =
    0
    \quad\text{in }\mathcal D'(\Omega).
\]
Equivalently, with
\(
    H_0:=m^{-1}I+F,
\)
we have
\[
    H_0\nabla\bq+\bz\nabla\psi=0,
    \qquad\text{or}\qquad
    \nabla(H_0\bq+\bz\psi)=0 .
\]
Since \(\bq\) and \(\psi\) have zero mean, it follows that
\(
    H_0\bq+\bz\psi=0.
\)
Using the limiting Poisson equation
\(
    -\Delta\psi=\bz\cdot\bq,
\)
we obtain
\(
    -\Delta\psi
    +
    \bz^\top H_0^{-1}\bz\,\psi
    =
    0 .
\)
Testing by \(\psi\) gives \(\psi=0\), and hence \(\bq=0\) follows from \(H_0 \bq + \bz \psi =0\).

It remains to contradict the normalization. Since
\(\mathcal H_k=\varepsilon_k^2\), we have
\(
    1=
    \frac{\mathcal H(\bc^k|\bc^\infty)}{\varepsilon_k^2}.
\)
The elementary quadratic expansion
\[
\begin{aligned}
\frac1{\varepsilon_k^2}
\int_\Omega
\left[
(m+\varepsilon_k q_i^k)
\ln\frac{m+\varepsilon_k q_i^k}{m}
-(m+\varepsilon_k q_i^k)+m
\right]dx
\longrightarrow
\frac1{2m}\int_\Omega |q_i|^2\,dx
\end{aligned}
\]
holds whenever \(q_i^k\to q_i\) strongly in \(L^2(\Omega)\) and
\(m+\varepsilon_k q_i^k\ge0\). It follows from Taylor expansion on the set
\(|\varepsilon_k q_i^k|\le m/2\) and uniform integrability on its complement.
Applying this expansion to \eqref{eq:relative-entropy}, and using the strong
convergences above, yields
\[
    1
    =
    \frac1{2m}
    \sum_{i=n,p}\int_\Omega |q_i|^2\,dx
    +
    \frac12\int_\Omega \bq^\top F\bq\,dx
    +
    \frac12\int_\Omega|\nabla\psi|^2\,dx .
\]
This contradicts \(\bq=0\) and \(\psi=0\). The desired inequality follows.
\end{proof}

The next lemma rules out non-equilibrium zero-dissipation limits on bounded
energy sublevels.

\begin{lem}[Zero-dissipation rigidity]
\label{lem:zero-dissipation-rigidity}
Let \(\bc^k=(c_n^k,c_p^k)^\top\) be admissible states satisfying
\[
    E(\bc^k)\le E_0,
    \qquad
    \mathcal D(\bc^k)\to0 .
\]
Then, up to a subsequence,
\[
    \bc^k\to\bc^\infty
    \quad\text{strongly in }L^2(\Omega)^2,
    \qquad
    \phi^k\to0
    \quad\text{strongly in }H^1(\Omega).
\]
\end{lem}

\begin{proof}
The energy bound gives \(\bc^k\) bounded in \(L^2(\Omega)^2\) and
\(\phi^k\) bounded in \(H^1(\Omega)\). The weighted gradient estimate used
above gives
\[
    \int_\Omega
    \frac{|\nabla c_n^k|^2+|\nabla c_p^k|^2}
         {1+c_n^k+c_p^k}
    \,dx
    \le C(E_0).
\]
Hence \(\bc^k\) is bounded in \(W^{1,4/3}(\Omega)^2\), and therefore, up to a
subsequence,
\[
    \bc^k\to\bc
    \quad\text{strongly in }L^2(\Omega)^2
    \quad\text{and a.e. in }\Omega .
\]
Standard elliptic regularity of the Neumann Poisson problem yields
\[
    \phi^k\to\phi
    \quad\text{strongly in }H^1(\Omega).
\]

As in the proof of Lemma~\ref{lem:local-eep}, set
\(
    C^k=\operatorname{diag}(c_n^k,c_p^k),
    \;
    \boldsymbol g^k
    =
    \begin{pmatrix}
    c_n^k\nabla\mu_n^k\\[1mm]
    c_p^k\nabla\mu_p^k
    \end{pmatrix}.
\)
Then
\(
    \boldsymbol g^k
    =
    (I+C^kF)\nabla\bc^k
    +
    C^k\bz\nabla\phi^k.
\)
The dissipation bound implies
\(
    \|\boldsymbol g^k\|_{L^1(\Omega)}
    \le C\mathcal D(\bc^k)^{1/2}\to0 .
\)
Repeating the coefficient argument from Lemma~\ref{lem:local-eep}, we pass to
the limit and obtain
\[
    (I+CF)\nabla\bc
    +
    C\bz\nabla\phi
    =
    0
    \quad\text{in }\mathcal D'(\Omega),
    \qquad
    C:=\operatorname{diag}(c_n,c_p).
\]
Equivalently, for \(i=n,p\),
\[
    \nabla c_i
    +
    c_i\nabla\bigl(z_i\phi+(F\bc)_i\bigr)
    =
    0,
    \qquad
    z_n=-1,\quad z_p=1 .
\]
Since the map
\(
    \bc\mapsto (I+C(\bc)F)^{-1}C(\bc)
\)
is globally bounded on \([0,\infty)^2\), the inverted relation gives
\(
    \nabla\bc
    =
    -(I+CF)^{-1}C\bz\nabla\phi
    \in L^2(\Omega).
\)
Hence \(c_i\in H^1(\Omega)\). By the standard stationary drift
rigidity for
\[
    \nabla u+u\nabla V=0
    \quad\text{in }\mathcal D'(\Omega),
    \qquad
    u,V\in H^1(\Omega),
\]
on a connected domain, one has \(ue^V=\text{constant}\) a.e.  
Therefore
\[
    \ln c_i+z_i\phi+(F\bc)_i
    =
    \text{constant},
    \qquad i=n,p .
\]
Thus \(\bc\) is a critical point of \(E\) under the two mass constraints. By
Lemma~\ref{lem:equilibrium-unique}, we can identify
\(
    \bc=\bc^\infty,
    \;
    \phi=0 .
\)
Since every convergent subsequence has the same limit, the convergence holds
for the whole sequence.
\end{proof}

We now upgrade the local estimate to bounded energy sublevels.

\begin{thm}[Sublevel entropy-entropy production]
\label{thm:sublevel-eep}
Assume \textup{(A1)}--\textup{(A5)}. For every
\(E_0>E(\bc^\infty)\), there exists a constant \(\lambda(E_0)>0\) such that
every admissible state satisfying
\(
    E(\bc)\le E_0
\)
obeys
\(
    \mathcal D(\bc)
    \ge
    \lambda(E_0)\mathcal H(\bc|\bc^\infty).
\)
\end{thm}

\begin{proof}
Suppose the assertion is false, then there exists a sequence of admissible states \(\bc^k\) such
that
\[
    E(\bc^k)\le E_0,
    \qquad
    \mathcal D(\bc^k)
    \le
    \frac1k\mathcal H(\bc^k|\bc^\infty).
\]
If
\(
    \mathcal H(\bc^k|\bc^\infty)\to0,
\)
then Lemma~\ref{lem:local-eep} applies for all sufficiently large \(k\), and
gives
\[
    \mathcal D(\bc^k)
    \ge
    \lambda_0\mathcal H(\bc^k|\bc^\infty),
\]
contradicting the preceding inequality.

Hence, up to a subsequence, there exists \(\eta>0\) such that
\(
    \mathcal H(\bc^k|\bc^\infty)\ge\eta .
\)
Since \(E(\bc^k)\le E_0\), the relative entropy is uniformly bounded, and
therefore
\(
    \mathcal D(\bc^k)\to0.
\)
Lemma~\ref{lem:zero-dissipation-rigidity} yields
\[
    \bc^k\to\bc^\infty
    \quad\text{strongly in }L^2(\Omega)^2,
    \qquad
    \phi^k\to0
    \quad\text{strongly in }H^1(\Omega).
\]
The steric and electrostatic parts of
\(\mathcal H(\bc^k|\bc^\infty)\) converge to zero. The logarithmic part also
converges to zero: the integrand
\(
    s\ln\frac{s}{m}-s+m
\)
is bounded above by \(C(1+s^2)\) on \([0,\infty)\), while
\(c_i^k\to m\) strongly in \(L^2(\Omega)\). Hence, the logarithmic integrands
are uniformly integrable and converge a.e. to zero. Thus
\(
    \mathcal H(\bc^k|\bc^\infty)\to0,
\)
contradicting \(\mathcal H(\bc^k|\bc^\infty)\ge\eta\). This completes the proof of the theorem.
\end{proof}


\subsection{Decay via the entropy approximation}
\label{subsec:decay-approximation}
We now apply Theorem~\ref{thm:sublevel-eep} to the mass-preserving entropy
approximation in the pure Neumann setting. Let
\(\bc_\tau^k=(c_{n,\tau}^k,c_{p,\tau}^k)^\top\) denote the discrete states
constructed by the entropy-variable scheme, and let \(t_k=k\tau\). In the
pure Neumann equal-mass case, the scheme preserves the two masses in
\textup{(A5)}. Moreover, the discrete entropy inequality gives
\begin{equation}
    E(\bc_\tau^k)
    +
    \tau\mathcal D(\bc_\tau^k)
    \le
    E(\bc_\tau^{k-1}),
    \qquad k\ge1,
\label{eq:discrete-energy-long-time}
\end{equation}
where \(\mathcal D(\bc_\tau^k)\) is the smooth entropy production
\eqref{eq:intro-dissipation-decomposition}. We choose the initial
approximation so that
\[
    \bc_\tau^0\to\bc_0
    \quad\text{strongly in }L^2(\Omega)^2,
    \qquad
    E(\bc_\tau^0)\to E(\bc_0).
\]

Fix \(E_*>E(\bc_0)\). For all sufficiently small \(\tau\), the discrete
energy inequality implies
\[
    E(\bc_\tau^k)\le E_*,
    \qquad k\ge0 .
\]
Hence, Theorem~\ref{thm:sublevel-eep} gives
\[
    \mathcal D(\bc_\tau^k)
    \ge
    \lambda(E_*)\mathcal H(\bc_\tau^k|\bc^\infty).
\]
Subtracting \(E(\bc^\infty)\) from
\eqref{eq:discrete-energy-long-time}, we obtain
\begin{equation}
    \mathcal H(\bc_\tau^k|\bc^\infty)
    +
    \tau\lambda(E_*)
    \mathcal H(\bc_\tau^k|\bc^\infty)
    \le
    \mathcal H(\bc_\tau^{k-1}|\bc^\infty).
\label{eq:discrete-H-decay-step}
\end{equation}
Iterating gives us
\begin{equation}
    \mathcal H(\bc_\tau^k|\bc^\infty)
    \le
    (1+\tau\lambda(E_*))^{-k}
    \mathcal H(\bc_\tau^0|\bc^\infty).
\label{eq:discrete-H-decay}
\end{equation}

Let \(\bc_\tau\) be the piecewise constant interpolation in time. Along the
same subsequence used in Section~\ref{sec:global-weak}, we have
\[
    \bc_\tau(t)\to\bc(t)
    \quad\text{strongly in }L^2(\Omega)^2
\]
for a.e. \(t>0\). By the lower semicontinuity of the relative entropy and the
convergence of the initial energies, passing to the limit in
\eqref{eq:discrete-H-decay} gives
\begin{equation}
    \mathcal H(\bc(t)|\bc^\infty)
    \le
    e^{-\lambda(E_*)t}
    \mathcal H(\bc_0|\bc^\infty)
    \qquad\text{for a.e. }t>0 .
\label{eq:weak-H-decay}
\end{equation}
Since \(E_*>E(\bc_0)\) is arbitrary, this proves the following result.

\begin{thm}[Exponential decay of approximation-generated weak solutions]
\label{thm:approx-generated-decay}
Assume \textup{(A1)}--\textup{(A5)}. Let \((\bc,\phi)\) be a finite-energy
weak solution obtained as a limit of the mass-preserving entropy
approximation in the pure Neumann setting. Then, for every
\(E_*>E(\bc_0)\),
\begin{equation}
    \mathcal H(\bc(t)|\bc^\infty)
    \le
    e^{-\lambda(E_*)t}
    \mathcal H(\bc_0|\bc^\infty)
    \qquad\text{for a.e. }t>0 .
\label{eq:approx-generated-decay}
\end{equation}
Consequently, along a set of times of full measure,
\[
    \bc(t)\to\bc^\infty
    \quad\text{strongly in }L^2(\Omega)^2,
    \qquad
    \phi(t)\to0
    \quad\text{strongly in }H^1(\Omega)
\]
as \(t\to\infty\).
\end{thm}

\begin{remark}[Entropy production at finite energy]
\label{rem:dissipation-defect}
The approximation qualification in
Theorem~\ref{thm:approx-generated-decay} is not due to a loss of
constitutive information in the finite-energy weak formulation. Indeed,
the fields \(\boldsymbol A_n,\boldsymbol A_p,\boldsymbol B\) are
intrinsically determined by \((\bc,\phi)\) through
\eqref{eq:weak-weighted-entropy-gradients} and
\eqref{eq:weak-collective-field}. For a finite-energy weak solution, define
its finite-energy entropy production by
\[
    \mathcal D_{\rm fe}(t)
    :=
    \int_\Omega
    \bigl(
        |\boldsymbol A_n(t)|^2
        +
        |\boldsymbol A_p(t)|^2
        +
        |\boldsymbol B(t)|^2
    \bigr)\,dx .
\]
The energy inequality controls \(\mathcal D_{\rm fe}\), but it is not known
whether this quantity controls the lower-semicontinuous extension of the
smooth entropy production needed in the sublevel
entropy-entropy production inequality.

More precisely, for \(E_*>E(\bc^\infty)\), define
\[
    \mathcal D_{E_*}^{\,*}(\bv)
    :=
    \inf
    \left\{
        \liminf_{j\to\infty}\mathcal D(\bv^j)
        :
        \begin{array}{l}
        \bv^j \text{ is admissible in the sense of }
        \textup{Subsection}~\ref{subsec:sublevel-eep},\\
        \bv^j\to\bv \text{ strongly in }L^2(\Omega)^2,\quad
        E(\bv^j)\le E_*
        \end{array}
    \right\}.
\]
By Theorem~\ref{thm:sublevel-eep},
\[
    \mathcal D_{E_*}^{\,*}(\bv)
    \ge
    \lambda(E_*)\mathcal H(\bv|\bc^\infty)
\]
whenever the right-hand side is finite. Therefore, if a finite-energy weak
solution satisfies the no-defect condition
\[
    \mathcal D_{\rm fe}(t)
    \ge
    \mathcal D_{E_*}^{\,*}(\bc(t))
    \qquad\text{for a.e. }t>0,
\]
together with the strong energy inequality
\[
    \mathcal H(\bc(t)|\bc^\infty)
    +
    \int_s^t \mathcal D_{\rm fe}(r)\,dr
    \le
    \mathcal H(\bc(s)|\bc^\infty)
\]
for a.e. \(0<s<t\), then the same entropy-entropy production argument gives
\[
    \mathcal H(\bc(t)|\bc^\infty)
    \le
    e^{-\lambda(E_*)(t-s)}
    \mathcal H(\bc(s)|\bc^\infty)
    \qquad\text{for a.e. }0<s<t .
\]
Thus, the remaining obstruction for a general finite-energy weak solution
is a possible gap between \(\mathcal D_{\rm fe}\) and the
lower-semicontinuous entropy production
\(\mathcal D_{E_*}^{\,*}\), rather than an ambiguity in the constitutive
identification.
\end{remark}

\subsection{Sharp linearized rate}
\label{subsec:linearized-rate}

We finally identify the sharp entropy-production constant in the small-sublevel
limit. This gives the precise linearized rate behind the sublevel
entropy-entropy production inequality proved above.

For \(E_0>E(\bc^\infty)\), define
\begin{equation}
\Lambda(E_0)
:=
\inf
\left\{
\frac{\mathcal D(\bc)}
     {\mathcal H(\bc|\bc^\infty)}
: \begin{array}{l}
\bc \text{ is admissible in the sense of Section}
~\ref{subsec:sublevel-eep},
E(\bc)\le E_0,\;\bc\ne\bc^\infty
\end{array}
\right\}.
\label{eq:optimal-sublevel-constant}
\end{equation}
By Theorem~\ref{thm:sublevel-eep}, \(\Lambda(E_0)>0\) for every
\(E_0>E(\bc^\infty)\). Moreover, \(E_0\mapsto\Lambda(E_0)\) is nonincreasing
since the admissible class enlarges as \(E_0\) increases.  Hence, the one-sided limit
\(
    \lim_{E_0\downarrow E(\bc^\infty)}\Lambda(E_0)
\)
exists, possibly as \(+\infty\). We show that this limit is in fact finite and is 
given by an explicit linearized Rayleigh quotient.

Recall that
\(
    H_0:=m^{-1}I+F,
    \;
    M_0:=M(\bc^\infty),
    \;
    \bz=(-1,1)^\top .
\)
For a zero-mean perturbation
\(
    \bq=(q_n,q_p)^\top,
    \;
    \int_\Omega q_n\,dx
    =
    \int_\Omega q_p\,dx
    =
    0,
\)
let \(\psi\) be the normalized Neumann solution of
\(
    -\Delta\psi=\bz\cdot\bq,
    \;
    \partial_\nu\psi=0,
    \;
    \int_\Omega\psi\,dx=0.
\)
Set
\(
    \boldsymbol\Theta:=H_0\bq+\bz\psi .
\)
The linearized relative entropy and the linearized entropy production are
defined by
\begin{equation}
    \mathcal H_{\rm lin}(\bq)
    :=
    \frac12
    \int_\Omega
    \bq^\top H_0\bq\,dx
    +
    \frac12
    \int_\Omega
    |\nabla\psi|^2\,dx ,
\label{eq:linearized-entropy}
\end{equation}
and
\begin{equation}
    \mathcal D_{\rm lin}(\bq)
    :=
    \int_\Omega
    \nabla\boldsymbol\Theta^\top
    M_0
    \nabla\boldsymbol\Theta\,dx .
\label{eq:linearized-dissipation}
\end{equation}
We then define
\begin{equation}
    \lambda_{\rm lin}
    :=
    \inf_{\substack{\bq\ne0\\
    \int_\Omega q_n\,dx=\int_\Omega q_p\,dx=0}}
    \frac{\mathcal D_{\rm lin}(\bq)}
         {\mathcal H_{\rm lin}(\bq)} .
\label{eq:lambda-lin-definition}
\end{equation}

\begin{lem}[Spectral formula for the linearized rate]
\label{lem:linearized-spectral-formula}
Let \(0=\nu_0<\nu_1\le\nu_2\le\cdots\) be the Neumann eigenvalues
of \(-\Delta\), repeated with multiplicity, and let
\[
    A_k:=H_0+\nu_k^{-1}\bz\otimes\bz,
    \qquad k\ge1.
\]
Then
\begin{equation}
    \lambda_{\rm lin}
    =
    2\inf_{k\ge1}
   \Big ( \nu_k\,
    \lambda_{\min}\!\left(M_0A_k\right) \Big ),
\label{eq:lambda-lin-spectral-formula}
\end{equation}
where \(\lambda_{\min}(M_0A_k)\) denotes the smallest eigenvalue of the
symmetric positive definite matrix
\(
    A_k^{1/2}M_0A_k^{1/2}.
\)
\end{lem}

\begin{proof}
Let \(\{e_k\}_{k\ge0}\) be an orthonormal Neumann eigenbasis, with \(e_0\) as a
constant. Since \(\bq\) has componentwise zero mean, it has the expansion
\(
    \bq=\sum_{k\ge1}\bq_k e_k,
    \;
    \bq_k\in\mathbb R^2.
\)
The associated potential satisfies
\(
    \psi
    =
    \sum_{k\ge1}
    \frac{\bz\cdot\bq_k}{\nu_k}e_k .
\)
Therefore
\(
    \boldsymbol\Theta
    =
    \sum_{k\ge1} A_k\bq_k e_k .
\)
Using orthogonality of the Neumann eigenfunctions,
\[
    \mathcal H_{\rm lin}(\bq)
    =
    \frac12
    \sum_{k\ge1}
    \bq_k^\top A_k\bq_k ,
\qquad
    \mathcal D_{\rm lin}(\bq)
    =
    \sum_{k\ge1}
    \nu_k\,
    \bq_k^\top A_kM_0A_k\bq_k .
\]
For each \(k\),
\[
    \frac{
    \bq_k^\top A_kM_0A_k\bq_k
    }{
    \bq_k^\top A_k\bq_k
    }
    \ge
    \lambda_{\min}(M_0A_k),
\]
where the eigenvalue is characterized by the symmetric representative
\(A_k^{1/2}M_0A_k^{1/2}\). This gives the lower bound in
\eqref{eq:lambda-lin-spectral-formula}. The reverse inequality follows by
testing the Rayleigh quotient with a single Neumann mode \(e_k\) and a vector
\(\xi\in\mathbb R^2 \setminus \{0\}\), and then optimizing over \(k\) and
\(\xi\).
\end{proof}

The next lemma is the lower-semicontinuity step needed in the sharp
small-sublevel limit. It is the only step where the nonlinear entropy
production is compared directly with its linearization.

\begin{lem}[Linearized lower semicontinuity of the entropy production]
\label{lem:linearized-dissipation-liminf}
Let \(\{\bc^j\}_{j\ge1}\) be admissible states in the sense of
Subsection~\ref{subsec:sublevel-eep} such that
\[
    E(\bc^j)\downarrow E(\bc^\infty),
    \qquad
    \bc^j\ne\bc^\infty.
\]
Set
\(
    \varepsilon_j^2
    :=
    \mathcal H(\bc^j|\bc^\infty),
    \;
    \bq^j
    :=
    \frac{\bc^j-\bc^\infty}{\varepsilon_j},
    \;
    \psi^j
    :=
    \frac{\phi^j}{\varepsilon_j},
\)
where \(\phi^j=\phi[\bc^j]\). Assume that
\begin{equation}
    \sup_j
    \frac{\mathcal D(\bc^j)}{\varepsilon_j^2}
    <\infty .
\label{eq:normalized-dissipation-bound}
\end{equation}
Then, up to a subsequence,
\[
    \bq^j\to\bq
    \quad\text{strongly in }L^2(\Omega)^2,
    \qquad
    \psi^j\to\psi
    \quad\text{strongly in }H^2(\Omega),
\]
where
\(
    -\Delta\psi=\bz\cdot\bq,
    \;
    \partial_\nu\psi=0,
    \;
    \int_\Omega\psi\,dx=0.
\)
Moreover,
\begin{equation}
    \mathcal H_{\rm lin}(\bq)=1,
\label{eq:linearized-entropy-normalization}
\end{equation}
and
\begin{equation}
    \liminf_{j\to\infty}
    \frac{\mathcal D(\bc^j)}{\varepsilon_j^2}
    \ge
    \mathcal D_{\rm lin}(\bq).
\label{eq:linearized-dissipation-liminf}
\end{equation}
\end{lem}

\begin{proof}
The entropy coercivity in \eqref{eq:relative-entropy-controls-L2} gives
\(
    \|\bq^j\|_{L^2(\Omega)}\le C .
\)
Moreover, the weighted gradient estimate used in
Lemma~\ref{lem:local-eep}, applied to the normalized sequence, gives
\[
    \int_\Omega
    \frac{|\nabla\bq^j|^2}{1+c_n^j+c_p^j}\,dx
    \le C .
\]
Since \(\bc^j\) is bounded in \(L^2(\Omega)^2\), H\"older's inequality yields
\[
    \|\bq^j\|_{W^{1,4/3}(\Omega)}\le C .
\]
Since \(W^{1,4/3}(\Omega)\hookrightarrow L^2(\Omega)\) compactly for
\(d\le3\), we may assume, after passing to a subsequence, that
\[
    \bq^j\to\bq
    \quad\text{strongly in }L^2(\Omega)^2 .
\]
The convergence of \(\psi^j\) in \(H^2(\Omega)\) follows from the normalized
Poisson equations
\(
    -\Delta\psi^j=\bz\cdot\bq^j,
    \;
    \partial_\nu\psi^j=0,
    \;
    \int_\Omega\psi^j\,dx=0,
\)
and the Neumann elliptic estimate.

The identity \eqref{eq:linearized-entropy-normalization} follows from the
quadratic expansion of the relative entropy at \(\bc^\infty\). Indeed, by the
same expansion used in the proof of Lemma~\ref{lem:local-eep},
\[
    \frac{\mathcal H(\bc^j|\bc^\infty)}{\varepsilon_j^2}
    \to
    \mathcal H_{\rm lin}(\bq).
\]
The left-hand side is identically one, so
\(\mathcal H_{\rm lin}(\bq)=1\).

It remains to prove \eqref{eq:linearized-dissipation-liminf}. Let
\(
    \boldsymbol\Theta^j
    :=
    \frac{\bmu^j-\bmu^\infty}{\varepsilon_j},
\)
where \(\bmu^j\) is the electrochemical potential associated with \(\bc^j\)
and \(\bmu^\infty\) is the constant equilibrium chemical potential. Since
\(\nabla\bmu^\infty=0\),
\[
    \frac{\mathcal D(\bc^j)}{\varepsilon_j^2}
    =
    \int_\Omega
    \nabla\boldsymbol\Theta^j{}^\top
    M(\bc^j)
    \nabla\boldsymbol\Theta^j\,dx .
\label{eq:normalized-dissipation-theta}
\]

Choose numbers \(\alpha_j\downarrow0\) such that $\frac{\|\bc^j-\bc^\infty\|_{L^2(\Omega)}}{\alpha_j}
    \to0,$ $\alpha_j<m/2,$ and define the  sets
\(
    G_j
    :=
    \left\{
    x\in\Omega:
    |c_n^j(x)-m|\le\alpha_j,\,
    |c_p^j(x)-m|\le\alpha_j
    \right\}.
\)
Then \(|\Omega\setminus G_j|\to0\). On \(G_j\), the concentrations are
uniformly bounded away from zero and
\[
    M(\bc^j)\to M_0,
    \qquad
    H(\bc^j):=\operatorname{diag}\bigl((c_n^j)^{-1},(c_p^j)^{-1}\bigr)+F
    \to H_0
\]
uniformly. Set
\(
    \boldsymbol W^j
    :=
    \mathbf 1_{G_j}\nabla\boldsymbol\Theta^j,
\)
where \(\mathbf 1_{G_j}\) is the characteristic function of the set \(G_j\). By \eqref{eq:normalized-dissipation-bound} and the uniform positivity of
\(M(\bc^j)\) on \(G_j\), the sequence \(\boldsymbol W^j\) is bounded in
\(L^2(\Omega;\mathbb R^{2\times d})\). Hence, up to a subsequence,
\[
    \boldsymbol W^j\rightharpoonup \boldsymbol W
    \quad\text{weakly in }L^2(\Omega;\mathbb R^{2\times d}) .
\]

We identify the weak limit \(\boldsymbol W\). Since
\(
    \nabla\boldsymbol\Theta^j
    =
    H(\bc^j)\nabla\bq^j+\bz\nabla\psi^j 
\)
on \(G_j\), while \(H(\bc^j)\to H_0\) uniformly on \(G_j\),
\(\bq^j\rightharpoonup\bq\) in \(W^{1,4/3}(\Omega)^2\), and
\(\psi^j\to\psi\) strongly in \(H^2(\Omega)\), it follows that, for every
smooth test field \(\boldsymbol\Xi\),
\[
    \int_\Omega
    \boldsymbol W^j:\boldsymbol\Xi\,dx
    \to
    \int_\Omega
    \nabla(H_0\bq+\bz\psi):\boldsymbol\Xi\,dx .
\]
Therefore
\(
    \boldsymbol W
    =
    \nabla(H_0\bq+\bz\psi)
    =
    \nabla\boldsymbol\Theta .
\)

Using the nonnegativity of the entropy-production density and restricting to
\(G_j\), we obtain from \eqref{eq:normalized-dissipation-theta} that
\[
\begin{aligned}
\liminf_{j\to\infty}
\frac{\mathcal D(\bc^j)}{\varepsilon_j^2}
\ge
\liminf_{j\to\infty}
\int_{G_j}
\nabla\boldsymbol\Theta^j{}^\top
M(\bc^j)
\nabla\boldsymbol\Theta^j\,dx
=
\liminf_{j\to\infty}
\int_\Omega
\boldsymbol W^j{}^\top
M_0
\boldsymbol W^j\,dx .
\end{aligned}
\]
In the last equality we used the uniform convergence
\(M(\bc^j)\to M_0\) on \(G_j\) and the \(L^2\)-boundedness of
\(\boldsymbol W^j\). The weak lower semicontinuity of the positive quadratic
form associated with \(M_0\) gives
\[
\liminf_{j\to\infty}
\int_\Omega
\boldsymbol W^j{}^\top
M_0
\boldsymbol W^j\,dx
\ge
\int_\Omega
\nabla\boldsymbol\Theta^\top
M_0
\nabla\boldsymbol\Theta\,dx
=
\mathcal D_{\rm lin}(\bq).
\]
This proves \eqref{eq:linearized-dissipation-liminf}.
\end{proof}

\begin{thm}[Sharp small-sublevel limit]
\label{thm:sharp-sublevel-limit}
Under the assumptions of Theorem~\ref{thm:sublevel-eep},
\begin{equation}
    \lim_{E_0\downarrow E(\bc^\infty)}
    \Lambda(E_0)
    =
    \lambda_{\rm lin}.
\label{eq:sharp-sublevel-limit}
\end{equation}
\end{thm}

\begin{proof}
We first prove the lower bound. Suppose that there exist \(\eta>0\), energy
levels \(E_j\downarrow E(\bc^\infty)\), and admissible states \(\bc^j\) such
that
\(
    E(\bc^j)\le E_j,
    \;
    \bc^j\ne\bc^\infty,
\)
and
\(
    \frac{\mathcal D(\bc^j)}
         {\mathcal H(\bc^j|\bc^\infty)}
    \le
    \lambda_{\rm lin}-\eta .
\)
Since \(E(\bc^j)\le E_j\downarrow E(\bc^\infty)\), we have
\(
    \varepsilon_j^2
    :=
    \mathcal H(\bc^j|\bc^\infty)
    \to0 .
\)
Moreover,
\(
    \frac{\mathcal D(\bc^j)}{\varepsilon_j^2}
    \le
    \lambda_{\rm lin}-\eta ,
\)
so Lemma~\ref{lem:linearized-dissipation-liminf} applies. Passing to a
subsequence, we obtain a nonzero perturbation \(\bq\) satisfying
\[
    \mathcal H_{\rm lin}(\bq)=1
\qquad  \text{ and } \qquad 
    \liminf_{j\to\infty}
    \frac{\mathcal D(\bc^j)}{\varepsilon_j^2}
    \ge
    \mathcal D_{\rm lin}(\bq).
\]
Therefore
\[
\begin{aligned}
\liminf_{j\to\infty}
\frac{\mathcal D(\bc^j)}
     {\mathcal H(\bc^j|\bc^\infty)}
=
\liminf_{j\to\infty}
\frac{\mathcal D(\bc^j)}{\varepsilon_j^2}
\ge
\mathcal D_{\rm lin}(\bq)
=
\frac{\mathcal D_{\rm lin}(\bq)}
     {\mathcal H_{\rm lin}(\bq)}
\ge
\lambda_{\rm lin},
\end{aligned}
\]
which contradicts the choice of \(\bc^j\). Hence
\[
    \liminf_{E_0\downarrow E(\bc^\infty)}
    \Lambda(E_0)
    \ge
    \lambda_{\rm lin}.
\]

We now prove the upper bound. Fix \(\eta>0\). By
Lemma~\ref{lem:linearized-spectral-formula}, there exist \(k\ge1\) and
\(\xi\in\mathbb R^2\setminus\{0\}\) such that
\[
    \frac{\mathcal D_{\rm lin}(\xi e_k)}
         {\mathcal H_{\rm lin}(\xi e_k)}
    \le
    \lambda_{\rm lin}+\eta .
\]
The mode \(k\) may be chosen finite because
\(\nu_k\to\infty\) and the matrices \(A_k\) are uniformly bounded from
below by \(H_0\), so the quantities in
\eqref{eq:lambda-lin-spectral-formula} diverge as \(k\to\infty\).

For sufficiently small \(\varepsilon>0\), define
\(
    \bc^\varepsilon
    :=
    \bc^\infty+\varepsilon \xi e_k .
\)
Since \(e_k\) has zero mean, \(\bc^\varepsilon\) satisfies the mass
constraints. For \(\varepsilon\) small enough, \(\bc^\varepsilon\) is admissible.
Let \(\phi^\varepsilon=\phi[\bc^\varepsilon]\). Then
\(
    \frac{\phi^\varepsilon}{\varepsilon}
    =
    \frac{\bz\cdot\xi}{\nu_k}e_k .
\)
By Taylor expansion of the entropy and of the entropy production at
\(\bc^\infty\),
\[
    \mathcal H(\bc^\varepsilon|\bc^\infty)
    =
    \varepsilon^2
    \mathcal H_{\rm lin}(\xi e_k)
    +
    o(\varepsilon^2),
\qquad
    \mathcal D(\bc^\varepsilon)
    =
    \varepsilon^2
    \mathcal D_{\rm lin}(\xi e_k)
    +
    o(\varepsilon^2).
\]
Thus
\[
    \frac{\mathcal D(\bc^\varepsilon)}
         {\mathcal H(\bc^\varepsilon|\bc^\infty)}
    \to
    \frac{\mathcal D_{\rm lin}(\xi e_k)}
         {\mathcal H_{\rm lin}(\xi e_k)}
    \le
    \lambda_{\rm lin}+\eta .
\]
Moreover,
\[
    E(\bc^\varepsilon)\downarrow E(\bc^\infty)
    \qquad
    \text{as }\varepsilon\downarrow0.
\]
Therefore, taking \(\bc = \bc^\varepsilon\) in
\eqref{eq:optimal-sublevel-constant}, we get
\[
    \limsup_{E_0\downarrow E(\bc^\infty)}
    \Lambda(E_0)
    \le
    \lambda_{\rm lin}+\eta .
\]
Letting \(\eta\downarrow0\) gives
\[
    \limsup_{E_0\downarrow E(\bc^\infty)}
    \Lambda(E_0)
    \le
    \lambda_{\rm lin}.
\]
Together with the lower bound, this proves
\eqref{eq:sharp-sublevel-limit}.
\end{proof}

\subsection{Strong stability near equilibrium}
\label{subsec:small-stability}

In this subsection, we study the long-time dynamics of smooth solutions as a consequence of Theorem~\ref{thm:sharp-sublevel-limit}. The constant
\(\lambda_{\rm lin}\), given in (\ref{eq:lambda-lin-spectral-formula}) and identified there as the sharp small-sublevel
entropy-production constant, also governs the local nonlinear relaxation of
strong solutions near the homogeneous equilibrium. Let \(-\Delta_N\) denote the Neumann Laplacian on the zero-mean subspace and,
for \(s\ge0\), define
\(    \|u\|_{H_N^s}
    :=
    \|(I-\Delta_N)^{s/2}u\|_{L^2(\Omega)} .
\)
For vector-valued functions, this norm is understood componentwise.

\begin{cor}[Nonlinear stability near equilibrium]
\label{cor:small-data-strong-stability}
Assume \textup{(A1)}--\textup{(A5)}, and assume in addition that
\(\partial\Omega\) is smooth. Let \(s>d/2+2\). For every
\(0<\gamma<\lambda_{\rm lin}\), there exist constants
\(\varepsilon_s>0\) and \(C_s>0\) such that, if
\(\bc_0-\bc^\infty\in H_N^s(\Omega)^2\) satisfies
\(
    \int_\Omega(c_{n,0}-m)\,dx
    =
    \int_\Omega(c_{p,0}-m)\,dx
    =
    0
\)
and
\(
    \|\bc_0-\bc^\infty\|_{H_N^s}
    \le
    \varepsilon_s,
\)
then the strong solution to (\ref{eq:intro-system-c})-(\ref{eq:intro-system-phi}) exists globally, remains strictly
positive, and satisfies
\begin{equation}
    \|\bc(t)-\bc^\infty\|_{H_N^s}^2
    +
    \|\phi(t)\|_{H_N^{s+2}}^2
    \le
    C_s e^{-\gamma t}
    \|\bc_0-\bc^\infty\|_{H_N^s}^2
    \qquad\text{for all }t\ge0 .
\label{eq:small-data-strong-decay}
\end{equation}
\end{cor}

\begin{proof}
Set
\(
    \bq:=\bc-\bc^\infty.
\)
The componentwise mass constraints imply
\(
    \int_\Omega q_n\,dx
    =
    \int_\Omega q_p\,dx
    =
    0,
\)
and the associated normalized Neumann potential satisfies
\begin{equation}
    -\Delta_N\phi=\bz\cdot\bq,
    \qquad
    \int_\Omega\phi\,dx=0.
\label{eq:strong-poisson}
\end{equation}
Equivalently,
\(
    \phi=\mathcal P_N(\bz\cdot\bq),
\)
where \(\mathcal P_N=(-\Delta_N)^{-1}\) denotes the normalized Neumann
Poisson operator.

For \(\bc\) in a sufficiently small \(H_N^s\)-neighborhood of
\(\bc^\infty\), the embedding
\(
    H_N^s(\Omega)\hookrightarrow W^{2,\infty}(\Omega)
\)
implies
\(
    c_n,c_p\ge \frac m2.
\)
In this neighborhood, the mobility \(M(\bc)\) and the entropy Hessian
\(
    H(\bc)
    :=
    \operatorname{diag}(c_n^{-1},c_p^{-1})+F
\)
depend smoothly on \(\bc\) and are uniformly positive definite. Moreover,
the no-flux boundary condition is equivalent to the componentwise
homogeneous Neumann condition for \(\bq\). Indeed, since \(M(\bc)\) is
invertible and \(\partial_\nu\phi=0\), the condition
\(
    M(\bc)\nabla\bmu\,\nu=0
\)
implies
\(
    0
    =
    \partial_\nu\bmu
    =
    H(\bc)\partial_\nu\bq,
\)
and hence \(\partial_\nu\bq=0\).

Eliminating \(\phi\) through \eqref{eq:strong-poisson}, the system can be
written on the componentwise zero-mean subspace in the quasilinear form
\begin{equation}
    \partial_t\bq+\mathcal A(\bq)\bq=0,
\label{eq:strong-quasilinear-form}
\end{equation}
where, for a vector-valued function \(\bv\),
\begin{equation}
\begin{aligned}
    \mathcal A(\bq)\bv
    :=
    -\nabla\cdot
    \Bigl[
        M(\bc^\infty+\bq)
        \Bigl(
            H(\bc^\infty+\bq)\nabla\bv
            +
            \bz\nabla\mathcal P_N(\bz\cdot\bv)
        \Bigr)
    \Bigr].
\end{aligned}
\label{eq:strong-quasilinear-operator}
\end{equation}
The linearization at \(\bq=0\) is
\[
    \mathcal A_0\bv
    :=
    -\nabla\cdot
    \Bigl[
        M_0
        \nabla
        \Bigl(
            H_0\bv
            +
            \bz\mathcal P_N(\bz\cdot\bv)
        \Bigr)
    \Bigr],
\]
where
\(
    M_0:=M(\bc^\infty),
    \;
    H_0:=m^{-1}I+F.
\)
Equivalently,
\begin{equation}
    \partial_t\bq
    =
    \mathcal L\bq+\mathcal Q(\bq),
    \qquad
    \mathcal L:=-\mathcal A_0,
    \qquad
    \mathcal Q(\bq)
    :=
    -\bigl(\mathcal A(\bq)-\mathcal A_0\bigr)\bq .
\label{eq:strong-abstract-form}
\end{equation}

The principal diffusion matrix \(M(\bc)H(\bc)\) is similar to the symmetric
positive definite matrix
\(
    H(\bc)^{1/2}M(\bc)H(\bc)^{1/2}.
\)
It is therefore uniformly normally elliptic in a sufficiently small
neighborhood of \(\bc^\infty\); hence the system is quasilinear parabolic in
the sense of \cite{amann1995linear}. The nonlocal term involving
\(\mathcal P_N\) is of lower order and does not affect normal ellipticity.

Since \(s>d/2+2\), the space \(H_N^s(\Omega)\) is an algebra and controls the
coefficients in \(W^{2,\infty}(\Omega)\). Standard Moser and composition
estimates, together with the boundedness of the Neumann Poisson map
\(
    \mathcal P_N:
    H_N^{s-2}(\Omega)
    \longrightarrow
    H_N^s(\Omega),
\)
show that
\(
    \bq\longmapsto\mathcal A(\bq)
\)
is \(C^1\) from a neighborhood of zero in \(H_N^s(\Omega)^2\) into
\(
    \mathcal L
    \bigl(
        H_N^s(\Omega)^2,
        H_N^{s-2}(\Omega)^2
    \bigr).
\)
The equivalent remainder in \eqref{eq:strong-abstract-form} satisfies
\(
    \mathcal Q(0)=0,
    \;
    D\mathcal Q(0)=0,
\)
and, for sufficiently small \(\bq,\bq_1,\bq_2\),
\begin{equation}
\begin{aligned}
    \|\mathcal Q(\bq)\|_{H_N^{s-2}}
    &\le
    C\|\bq\|_{H_N^s}^2,\\
    \|\mathcal Q(\bq_1)-\mathcal Q(\bq_2)\|_{H_N^{s-2}}
    &\le
    C
    \bigl(
        \|\bq_1\|_{H_N^s}
        +
        \|\bq_2\|_{H_N^s}
    \bigr)
    \|\bq_1-\bq_2\|_{H_N^s}.
\end{aligned}
\label{eq:strong-quadratic-remainder}
\end{equation}
Normal ellipticity and the Neumann boundary conditions provide the required
maximal-regularity realization and a local \(C^1\) semiflow, together with
the corresponding continuation criterion; see
\cite{amann1995linear,lunardi2012analytic}. The Moser estimates used above
are standard; see \cite{taylor1996partial}.

We next determine the decay rate of the linearized problem. Let
\(\{e_k\}_{k\ge0}\) be an orthonormal Neumann eigenbasis satisfying
\(
    -\Delta_Ne_k=\nu_ke_k,
\)
and write
\(
    \bq=\sum_{k\ge1}\bq_k e_k.
\)
Equation~\eqref{eq:strong-poisson} gives
\(
    \phi
    =
    \sum_{k\ge1}
    \frac{\bz\cdot\bq_k}{\nu_k}e_k.
\)
With
\(
    A_k
    :=
    H_0+\nu_k^{-1}\bz\otimes\bz,
\)
the \(k\)-th mode of the linearized equation satisfies
\(
    \partial_t\bq_k
    =
    -\nu_kM_0A_k\bq_k.
\)

Define the higher-order linearized energy
\begin{equation}
    \mathcal H_{\rm lin}^{(s)}(\bq)
    :=
    \frac12
    \sum_{k\ge1}
    (1+\nu_k)^s
    \bq_k^\top A_k\bq_k.
\label{eq:higher-linearized-entropy}
\end{equation}
Since
\(    H_0
    \le
    A_k
    \le
    H_0+\nu_1^{-1}\bz\otimes\bz,
    \; k\ge1,
\)
the functional \(\mathcal H_{\rm lin}^{(s)}\) is equivalent to
\(\|\bq\|_{H_N^s}^2\).

Along the linearized evolution,
\[
    \frac{d}{dt}\mathcal H_{\rm lin}^{(s)}(\bq)
    =
    -
    \sum_{k\ge1}
    (1+\nu_k)^s\nu_k
    \bq_k^\top A_kM_0A_k\bq_k.
\]
For every \(k\ge1\), the spectral formula
\eqref{eq:lambda-lin-spectral-formula} gives
\[
\begin{aligned}
    \nu_k\bq_k^\top A_kM_0A_k\bq_k
    \ge
    \nu_k\lambda_{\min}(M_0A_k)
    \bq_k^\top A_k\bq_k\ge
    \frac{\lambda_{\rm lin}}2
    \bq_k^\top A_k\bq_k.
\end{aligned}
\]
Consequently,
\begin{equation}
    \frac{d}{dt}\mathcal H_{\rm lin}^{(s)}(\bq)
    \le
    -\lambda_{\rm lin}
    \mathcal H_{\rm lin}^{(s)}(\bq).
\label{eq:higher-linearized-decay}
\end{equation}
The norm equivalence following
\eqref{eq:higher-linearized-entropy} therefore yields
\begin{equation}
    \|e^{t\mathcal L}\bq\|_{H_N^s}^2
    \le
    C e^{-\lambda_{\rm lin}t}
    \|\bq\|_{H_N^s}^2.
\label{eq:strong-linear-semigroup-decay}
\end{equation}
Thus, the linearized flow has decay rate
\(\lambda_{\rm lin}\) at the squared-norm level, or equivalently
\(\lambda_{\rm lin}/2\) at the norm level.

The componentwise mass constraints restrict the evolution to the zero-mean
subspace and remove the two constant mass directions. Hence
\(\bq=0\) is an isolated equilibrium on this subspace, and
\eqref{eq:strong-linear-semigroup-decay} implies
\(
    0\notin\sigma(\mathcal L).
\)
The quasilinear formulation
\eqref{eq:strong-quasilinear-form}, the \(C^1\)-dependence and normal
ellipticity of \(\mathcal A(\bq)\), and the spectral estimate
\eqref{eq:strong-linear-semigroup-decay} therefore verify the hypotheses of
the classical principle of linearized stability, corresponding to the
isolated-equilibrium case of
\cite[Theorem~5.3.1 and Remark~5.3.2(a)]{pruss2016moving}.

For every
\(
    0<\omega<\lambda_{\rm lin}/2,
\)
this principle gives \(\varepsilon_s>0\) such that
\[
    \|\bq(0)\|_{H_N^s}
    \le
    \varepsilon_s
\]
implies global existence and
\[
    \|\bq(t)\|_{H_N^s}
    \le
    C_s e^{-\omega t}
    \|\bq(0)\|_{H_N^s}.
\]
Taking \(\gamma=2\omega\), we obtain, for every
\(0<\gamma<\lambda_{\rm lin}\),
\begin{equation}
    \|\bq(t)\|_{H_N^s}^2
    \le
    C_s e^{-\gamma t}
    \|\bq(0)\|_{H_N^s}^2.
\label{eq:strong-concentration-decay}
\end{equation}
The resulting solution remains in the prescribed neighborhood of
\(\bc^\infty\). Choosing this neighborhood so that
\(
    \|\bq\|_{L^\infty}<m/2
\)
ensures
\(
    c_n,c_p\ge\frac m2
\)
for all \(t\ge0\). The continuation criterion therefore prevents exit from
the normally elliptic regime.

Finally, the Neumann elliptic estimate applied to
\eqref{eq:strong-poisson} gives
\[
    \|\phi(t)\|_{H_N^{s+2}}
    \le
    C\|\bq(t)\|_{H_N^s}.
\]
Combining this estimate with
\eqref{eq:strong-concentration-decay} proves
\eqref{eq:small-data-strong-decay}.
\end{proof}

\section{Conclusion}
\label{sec:conclusion}

The analysis in this paper is based on the observation that interspecies drag fundamentally alters the metric structure of the steric Poisson--Nernst--Planck system, rather than merely adding a perturbative transport term. In the energetic variational formulation, steric effects are encoded in the free energy, while drag enters through the dissipation and produces a concentration-dependent, non-diagonal Onsager mobility. This leads to dissipation modes that are not equivalent to the individual entropy gradients, and the finite-energy weak formulation has to be built at that level. The existence theory, the sublevel entropy-entropy production inequality, the sharp linearized rate, and its nonlinear realization through small-data strong stability can therefore be viewed as different uses of the same underlying energy-dissipation structure. 

The remaining difficulties are also tied to this structure. For general finite-energy weak solutions, it remains open whether the intrinsic finite-energy dissipation controls the lower-semicontinuous extension of the entropy production required by the sublevel inequality. The rank-one steric case shows a related obstruction: when the free energy controls only part of the density vector, the resulting lack of compactness occurs precisely in the modes required to identify the ionic fluxes. These issues suggest that future extensions to partially degenerate steric interactions or more general multicomponent drag laws will require a sharper understanding of the dissipation modes themselves.

\section*{Acknowledgments} 
The authors would like to thank Prof. Chia-Yu Hsieh (Sun Yat-sen University) for helpful discussions and guidance on analytical approaches to Poisson--Nernst--Planck systems with drags. They also thank Prof. Yiwei Wang (University of California, Riverside) for valuable suggestions on the modeling aspects of the work.  Baoli Hao and Chun Liu are partially supported by NSF DMS-2410742 and DMS-2118181.

\end{document}